\documentclass[10pt]{amsart}
\usepackage{amscd, amssymb}
\usepackage{pb-diagram}
\usepackage[all]{xypic}
\theoremstyle{plain}
\newtheorem{Thm}[equation]{Theorem}
\newtheorem{Cor}[equation]{Corollary}
\newtheorem{Prop}[equation]{Proposition}
\newtheorem{Lem}[equation]{Lemma}
\numberwithin{equation}{section}

\newcommand{\OO}{\operatorname{O}}
\newcommand{\GO}{\operatorname{GO}}

\newcommand{\GSO}{\operatorname{GSO}}
\newcommand{\SO}{\operatorname{SO}}
\newcommand{\PGSp}{\operatorname{PGSp}}
\newcommand{\GSp}{\operatorname{GSp}}
\newcommand{\GSpin}{\operatorname{GSpin}}

\newcommand{\Sp}{\operatorname{Sp}}
\newcommand{\Ind}{\operatorname{Ind}}

\newcommand{\Hom}{\operatorname{Hom}}

\newcommand{\GL}{\operatorname{GL}}

\newcommand{\SL}{\operatorname{SL}}

\newcommand{\C}{\mathbb C}
\newcommand{\A}{\mathbb{A}}

\newcommand{\Z}{\mathbb{Z}}
\newcommand{\R}{\mathbb{R}}

\newcommand{\bm}{\begin{multline*}}
\newcommand{\tu}{\end  {multline*}}

\newcommand{\simi}{\operatorname{sim}}

\title{The Local Langlands Conjecture for $\GSp(4)$}
\author{Wee Teck Gan and Shuichiro Takeda}
\address{Mathematics Department, University of California, San Diego, 9500 Gilman Drive, La Jolla,
92093}
\email{wgan@math.ucsd.edu}
\email{shtakeda@math.ucsd.edu}
\subjclass[2000]{11F27, 11F70,  22E50}

\begin{document}
\maketitle

\begin{abstract}
 We prove the local Langlands conjecture for $\GSp_4(F)$ where $F$ is a non-archimedean local
 field of characteristic zero.
 \end{abstract}

\section{\bf Introduction}

Let $F$ be a non-archimedean local field of characteristic $0$ and residue characteristic $p$. Let
$W_F$ be the Weil group of $F$ and let $WD_F = W_F \times \SL_2(\C)$ be the Weil-Deligne group. It was shown by Harris-Taylor \cite{HT} and Henniart \cite{He1} that there is a natural bijection between the set of equivalence classes of irreducible smooth representations of $\GL_n(F)$ and the set of conjugacy classes of $L$-parameters for $\GL_n$, i.e. admissible homomorphisms
\[  \phi: WD_F \longrightarrow \GL_n(\C). \]
This bijection satisfies a number of natural conditions which determine it uniquely.
\vskip 5pt

For a general connected reductive group $G$ over $F$, which we assume to be split for simplicity, Langlands  conjectures that there is a surjective finite-to-one map from the set $\Pi(G)$ of (equivalence classes of) irreducible smooth representations of $G(F)$ to the set $\Phi(G)$ of (equivalence classes of)  admissible homomorphisms
\[  WD_F \longrightarrow  G^{\vee} \]
where $G^{\vee}$ is the Langlands dual group of $G$ and the homomorphisms are taken up to
$G^{\vee}$-conjugacy. This leads to a partition of the set of equivalence classes of irreducible representations of $G(F)$ into a disjoint union of finite subsets, which are the fibers of the map and are called
 $L$-packets. Again, this map is supposed to preserve natural invariants which one can attach to both sides. These natural invariants are the $\gamma$-factors, $L$-factors and $\epsilon$-factors. Unfortunately, on the representation theoretic side,  one only has a general theory of these invariants for generic representations of $G(F)$.

\vskip 10pt

The purpose of this paper is to prove the local Langlands conjecture for $G = \GSp_4$.
\vskip 5pt

\noindent{\bf \underline{Main Theorem}}
\vskip 5pt

\noindent There is a {\em surjective} finite-to-one map
\[   L: \Pi(\GSp_4) \longrightarrow \Phi(\GSp_4) \]
with the following properties:

 \vskip 5pt
\noindent (i) $\pi$ is a (essentially) discrete series representation of $\GSp_4(F)$ if and only if its $L$-parameter $\phi_{\pi}:= L(\pi)$  does not factor through any proper Levi subgroup of $\GSp_4(\C)$.

\vskip 5pt

\noindent (ii) For an $L$-parameter $\phi$, its fiber $L_{\phi}$ can be naturally parametrized by the set of irreducible characters of the component group
\[  A_{\phi} = \pi_0(Z(Im(\phi)) / Z_{\GSp_4}). \]
This component group is either trivial or equal to $\Z/2\Z$. When $A_{\phi} =\Z/2\Z$, exactly one of the two representations in $L_{\phi}$ is generic and it is the one indexed by the trivial character of $A_{\phi}$.
\vskip 5pt

\noindent (iii)  The similitude character $\simi(\phi_{\pi})$ of $\phi_{\pi}$ is equal to the central character $\omega_{\pi}$ of $\pi$ (via local class field theory). Here, $\simi: \GSp_4(\C) \longrightarrow \C^{\times}$ is the similitude character of $\GSp_4(\C)$.
\vskip 10pt

\noindent (iv)   The $L$-parameter of $\pi \otimes (\chi \circ \lambda)$ is equal to
$\phi_{\pi} \otimes \chi$. Here, $\lambda : \GSp_4(F) \longrightarrow F^{\times}$ is the similitude character of $\GSp_4(F)$, and we have regarded $\chi$ as both a character of $F^{\times}$ and a character of $W_F$ by local class field theory.
\vskip 10pt

\noindent (v) Suppose that $\pi$ is a generic representation or a non-supercuspidal representation. Then for any irreducible representation $\sigma$ of $\GL_r(F)$, with $L$-parameter $\phi_{\sigma}$, we have:
\[  \begin{cases}
\gamma(s, \pi \times \sigma, \psi) = \gamma(s, \phi_{\pi} \otimes \phi_{\sigma}, \psi) \\
L(s, \pi \times \sigma) = L(s, \phi_{\pi} \otimes \phi_{\sigma}) \\
\epsilon(s, \pi \times \sigma, \psi) = \epsilon(s, \phi_{\pi} \otimes \phi_{\sigma}, \psi).\end{cases}  \]
Here the functions on the RHS are the local factors of Artin type associated to the relevant representations of $WD_F$, whereas those on the LHS are the local factors of Shahidi, as defined in \S \ref{S:certain}.
\vskip 10pt

\noindent (vi) Suppose that $\pi$ is a non-generic supercuspidal representation. For any irreducible supercuspidal representation $\sigma$ of $\GL_r(F)$ with $L$-parameter $\phi_{\sigma}$, let $\mu(s, \pi \boxtimes \sigma,\psi)$ denote the Plancherel measure associated to the family of induced representations
$I_P(\pi \boxtimes \sigma,s)$ on $\GSpin_{2r+5}(F)$, where we have regarded $\pi \boxtimes \sigma$ as a representation of the Levi subgroup $\GSpin_5(F) \times \GL_r(F)\cong \GSp_4(F) \times \GL_r(F)$. Then $\mu(s, \pi \boxtimes \sigma,\psi)$ is equal to
\[  \gamma(s,  \phi_{\pi}^{\vee} \otimes \phi_{\sigma}, \psi) \cdot
\gamma(-s, \phi_{\pi} \otimes \phi_{\sigma}^{\vee}, \overline{\psi}) \cdot
 \gamma(2s, Sym^2\phi_{\sigma}  \otimes \simi \phi_{\pi}^{-1},\psi) \cdot
\gamma(-2s, Sym^2 \phi_{\sigma}^{\vee} \otimes  \simi \phi_{\pi}, \overline{\psi}). \]
 \vskip 10pt

\noindent (vii) An $L$-packet $L_{\phi}$ contains a generic representation if and only if the adjoint $L$-factor $L(s , Ad \circ \phi)$ is holomorphic at $s = 1$. Here, $Ad$ denotes the adjoint representation of $\GSp_4(\C)$ on the complex Lie algebra $\mathfrak{sp}_4$. Moreover, $L_\phi$ contains an
essentially tempered generic representation if and only if $\phi$ is an essentially tempered $L$-parameter, i.e. $\phi|_{W_F}$ has bounded image in $\PGSp_4(\C)$. In this case, the generic representation in the packet is unique and is indexed by the trivial character of $A_{\phi}$.
\vskip 10pt

\noindent (viii) The map $L$ is uniquely determined by the properties (i), (iii), (v) and (vi), with $r \leq 2$  in (v) and (vi).
\hfill$\blacksquare$

\vskip 10pt

To the best of our knowledge, for non-generic supercuspidal representations, the theory of local $\gamma$-factors, $L$-factors and $\epsilon$-factors of pairs has not been fully developed and so at this point, it is not possible for us to say anything regarding these in part (v) of the theorem. However, if one assumes the existence of a theory of $\gamma$-factors satisfying the usual properties (such as those listed as the ``Ten Commandments" in \cite{LR}), then we can show that (v) holds for all representations, in which case the map $L$  will be uniquely characterized by (iii) and (v) (with $r \leq 2$ in (v)).
\vskip 10pt

In any case, we substitute the (as yet non-existent)  theory of $\gamma$-factors of pairs by the Plancherel measure.
The Plancherel measure in (vi) is a coarser invariant than the $\gamma$-factor, but has the advantage that it is defined for all representations. For generic representations, the identity in (vi) follows from results of Shahidi \cite{Sh} and Henniart \cite{He2}. Thus, (vi) shows that the Plancherel measure is an invariant  of a supercuspidal $L$-packet. Moreover, it turns out that this coarser invariant
is sufficient to distinguish between the non-generic supercuspidal representations of $\GSp_4(F)$, and this gives the characterization of $L$ by (i), (iii), (v) and (vi).

\vskip 10pt

The proof of the theorem relies on the local Langlands correspondence for $\GL_2$ and $\GL_4$ and a consideration of the following two towers of dual pairs:

\[
\xymatrix@R=2pt{
&&\GSO_{3,3}\\
&\GSp_4 \ar@{-}[ru]\ar@{-}[ld]\ar@{-}[rd]&\\
\GSO_{4,0}&&\GSO_{2,2}
}
\]

\vskip 5pt

\noindent Thus, it relies on a study of the local theta correspondence arising between $\GSp_4$ and the orthogonal similitude groups associated to quadratic spaces of rank $4$ or $6$ with trivial discriminant and the accidental isomorphisms:

\[  \begin{cases}
\GSO_{2,2}  \cong (\GL_2 \times \GL_2)/\{ (z,z^{-1}): z \in \GL_1 \} \\
\GSO_{3,3} \cong (\GL_4 \times \GL_1)/ \{ (z,z^{-2}): z \in \GL_1 \} \end{cases} \]
as well as the analogs for the inner forms.
\vskip 10pt

The first attempt to define $L$-packets for $\GSp_4$ is the paper \cite{V} of Vigneras, who considered the case $p \ne 2$.  She defined her $L$-packets via theta lifts from various forms of $\GO_4$ (including those whose discriminant  is not trivial). However, she did not check that her $L$-packets exhaust  $\Pi(\GSp_4)$.
In a more recent paper \cite{Ro2}, B. Roberts has given a more detailed treatment of Vigneras' construction of local $L$-packets (including the case $p=2$).

\vskip 10pt

Our construction of the local $L$-packets, on the other hand, relies only on the theta lifting from 
$\GSp_4$ to
$\GSO_{2,2} $, $\GSO_{4,0}$ and $\GSO_{3,3}$. It works for any $p$ and also
enables us to show that our packets exhaust all representations. We shall also show that, starting from a given $L$-parameter,  the $L$-packets constructed by Vigneras and Roberts agree with ours. As a consequence, we show that when $p \ne 2$, their construction exhausts all irreducible representations of $\GSp_4(F)$.   We also mention that, considering  only those representations of $\GSp_4(F)$ whose central character is a square, a section of our map $L$ was constructed by Jiang-Soudry \cite{JSo}. More precisely, for the split group $G =\SO_{2n+1}$, they constructed an injective (but definitely not surjective unless $n=1$) map from $\Phi(G) \longrightarrow \Pi(G)$ verifying  the conditions (i), (iv) and (v) of our main theorem.

\vskip 10pt

Let us mention the various key ingredients used in the proof of the theorem. The first is a paper of Mui\'{c}-Savin \cite{MS} which studies the theta lifting of generic discrete series representations for isometry groups and relates the non-vanishing of these theta lifts to the local $L$-functions of Shahidi. The second ingredient is  a paper of Kudla-Rallis \cite{KR} which proves the conservation conjecture for the first occurrences of supercuspidal representations of symplectic groups in orthogonal towers. In particular, their results imply that every non-generic supercuspidal representations of  $\GSp_4(F)$ can be obtained as a theta lift from the anisotropic $\GSO_{4,0}$. The third is a recent result of Henniart \cite{He2} which shows that the local Langlands correspondence for $\GL_n$ respects the twisted exterior square $L$-factor. Finally, we have a companion paper \cite{GT4} in which we determine completely the three local theta correspondences mentioned above. The detailed study of these theta correspondences is necessary to supplement the results of \cite{KR} and \cite{MS}, particularly for the non-discrete series representations.

 \vskip 10pt

The crucial results of \cite{KR}, \cite{MS} and \cite{He2} are reviewed in Section \ref{S:KRMS}, after we introduce some basic facts about the theta correspondence for similitudes in Sections \ref{S:theta} and \ref{S:dual}
and recall Shahidi's construction of certain $L$-functions in Section \ref{S:certain}.
Our construction of the $L$-packets and the proof of exhaustion are given in Section \ref{S:packets}, where we also verify the relation of genericity and the adjoint $L$-factor.  The preservation of local factors and Plancherel measures is demonstrated in Sections \ref{S:factors} and \ref{S:Plan} respectively,  and the characterization of $L$ is given in Section \ref{S:char}.
In Section \ref{S:vigneras}, we reconcile our construction with that given by Vigneras and Roberts. Finally, in Section \ref{S:global}, we give a global application: using the results of this paper, we establish the strong lifting of generic cuspidal representations from $\GSp_4$ to $\GL_4$.

 \vskip 10pt

We conclude this introduction with a number of subsequent developments:
 \vskip 5pt

\begin{itemize}
\item[(i)] In a sequel \cite{GT2} to this paper, we deduce the local Langlands correspondence for $\Sp_4$ from the results of this paper;
\vskip 5pt

\item[(ii)] In another sequel \cite{GTW} to this paper, by the first author and W. Tantono, the local Langlands correspondence is extended to the unique inner form of $\GSp_4$;
\vskip 5pt

\item[(iii)]  An $L$-packet is supposed to be ``stable" and  should satisfy some character identities relative to (twisted) endoscopic transfers.
Our method, unfortunately, does not shed any light on these harmonic analytic issues. However, another  sequel \cite{CG} to this paper, by P. S. Chan and the first author, establishes these properties of the $L$-packets constructed here using the Arthur-Selberg trace formula;
\vskip 5pt

\item[(iv)] In a recent paper \cite{DR}, Debacker and Reeder have given a construction of $L$-packets associated to certain tamely ramified $L$-parameters of an arbitrary reductive group $G$. The elements in their $L$-packets are all depth zero supercuspidal representations. One can ask whether their packets agree with ours in the case $G = \GSp_4$.  This is shown to be the case in the UCSD PhD thesis of J. Lust;

\item[(v)] A theory of $L$- and $\epsilon$-factors for $\GSp_4 \times  \GL_r$ (for $r = 1$ or $2$)  which is valid for all representations, including the non-generic ones, is being developed by N. Townsend in his UCSD PhD thesis.
 \end{itemize}

  \vskip 5pt
\noindent{\bf Acknowledgments:} We thank Dipendra Prasad, Brooks Roberts, Gordan Savin  and Marie-France Vigneras for their interests, suggestions and encouragements. We also thank Wilhelm Zink for sending us a copy of his letter \cite{Z} to Vigneras. W. T. Gan's research is partially supported by NSF grant DMS-0500781.

\vskip 15pt

\section{\bf Similitude Theta Correspondences}  \label{S:theta}

In this section, we shall describe some basic properties of the theta correspondence for similitude groups. The definitive reference for this subject matter is the paper [Ro1] of B. Roberts.
However, the results of [Ro1] are not sufficient for our purposes and need to be somewhat extended.

\vskip 10pt
Consider the dual pair $\OO(V) \times \Sp(W)$; for simplicity, we assume that $\dim V$ is even.
For each non-trivial additive character $\psi$, let $\omega_{\psi}$ be the Weil representation for $\OO(V) \times \Sp(W)$, which can be described as follows. Fix a Witt decomposition
$W = X \oplus Y$ and let $P(Y) = \GL(Y) \cdot N(Y)$ be the parabolic subgroup stabilizing the maximal isotropic subspace $Y$. Then
\[  N(Y) = \{ b \in Hom(X,Y) : b^t  = b \}, \]
where $b^t \in Hom(Y^*, X^*) \cong Hom(X,Y)$.
The Weil representation $\omega_{\psi}$  can be realized on the Schwartz space $S(X \otimes V)$ and the action of $P(Y) \times \OO(V)$ is
given by the usual formulas:
\vskip 5pt
\[  \begin{cases}
\omega_{\psi}(h)\phi(x) = \phi(h^{-1} x), \quad \text{for $h \in \OO(V)$;} \\
\omega_{\psi}(a)\phi(x) =  \chi_V(\det_Y(a)) \cdot |\det_Y(a)|^{\frac{1}{2} \dim V} \cdot \phi(a^{-1} \cdot x), \quad \text{for $a \in \GL(Y)$;} \\
\omega_{\psi}(b) \phi(x) = \psi( \langle bx, x \rangle) \cdot \phi(x), \quad \text{for $b \in N(Y)$,} \end{cases}
\]
\vskip 5pt
\noindent where  $\chi_V$ is the quadratic character associated to $\operatorname{disc} V \in F^{\times}/ F^{\times 2}$ and $\langle -, -\rangle$ is the natural symplectic form on $W \otimes V$.
To describe the full action of $\Sp(W)$, one needs to specify the action of a Weyl group element, which acts by a Fourier transform.

\vskip 10pt

If $\pi$ is an irreducible representation of $\OO(V)$ (resp. $\Sp(W)$), the maximal $\pi$-isotypic quotient has the form
\[  \pi \boxtimes  \Theta_{\psi}(\pi) \]
for some smooth representation $ \Theta_{\psi}(\pi)$ of $\Sp(W)$ (resp. $\OO(V)$). We call $\Theta_{\psi}(\pi)$ the big theta lift of $\pi$. It is known that $\Theta_{\psi}(\pi)$ is of finite length and hence is admissible. Let
$\theta_{\psi}(\pi)$ be the maximal semisimple quotient of $\Theta_{\psi}(\pi)$; we call it the small theta lift of $\pi$.  Then it was a conjecture of Howe that
\vskip 5pt

\begin{itemize}
\item $\theta_{\psi}(\pi)$ is irreducible whenever $\Theta_{\psi}(\pi)$ is non-zero.
\vskip 5pt
\item the map $\pi \mapsto \theta_{\psi}(\pi)$ is injective on its domain.
\end{itemize}
\vskip 5pt

\noindent This has been proved by Waldspurger when the residual characteristic $p$ of $F$ is not $2$ and can be checked in many low-rank cases, regardless of the residual characteristic of $F$. If the Howe conjecture is true in general, our treatment in the rest of the paper can be somewhat simplified. However, because we would like to include the case $p =2$ in our discussion, we shall refrain from assuming Howe's conjecture in this paper.
\vskip 10pt

With this in mind, we take note of the following result which was shown by Kudla \cite{K} for any residual characteristic $p$:
\vskip 5pt

\begin{Prop} \label{P:Kudla}
\noindent (i) If $\pi$ is supercuspidal, $\Theta_{\psi}(\pi) = \theta_{\psi}(\pi)$ is irreducible or zero.
\vskip 5pt

\noindent (ii) If $\theta_{\psi}(\pi_1) = \theta_{\psi}(\pi_2) \ne 0$ for two supercuspidal representations
$\pi_1$ and $\pi_2$, then $\pi_1 = \pi_2$.
\end{Prop}
\vskip 5pt

\noindent One of the main purposes of this section is to extend this result of Kudla to the case of similitude groups.

\vskip 5pt

Let $\lambda_V$ and $\lambda_W$ be the similitude factors of $\GO(V)$ and $\GSp(W)$ respectively.
We shall consider the group
\[  R = \GO(V) \times \GSp(W)^+ \]
where $\GSp(W)^+$ is the subgroup of $\GSp(W)$ consisting of elements $g$ such that $\lambda_W(g)$ is in the image of $\lambda_V$. In fact, for the cases of interest in this paper (see the next section),  $\lambda_V$ is surjective, in which case $\GSp(W)^+ = \GSp(W)$.
\vskip 5pt

The group $R$ contains the subgroup
\[  R_0 = \{ (h,g) \in R: \lambda_V(h) \cdot \lambda_W(g) = 1\}. \]
 The Weil representation $\omega_{\psi}$ extends naturally to the group $R_0$ via
 \[  \omega_{\psi}(g,h)\phi = |\lambda_V(h)|^{- \frac{1}{8}\dim V \cdot \dim W} \omega(g_1, 1)(\phi \circ h^{-1}) \]
where
\[  g_1 = g \left(  \begin{array}{cc}
\lambda(g)^{-1} & 0 \\
0 & 1  \end{array} \right) \in \Sp(W). \]
Note that this differs from the normalization used in \cite{Ro1}.
Observe in particular that the central elements $(t,t^{-1}) \in R_0$ act by the quadratic character
$\chi_V(t)^{\frac{\dim W}{2}}$.
 \vskip 5pt

Now consider the (compactly) induced representation
\[  \Omega = ind_{R_0}^R \omega_{\psi}. \]
As a representation of $R$, $\Omega$ depends only on the orbit of $\psi$ under the evident action of  $\operatorname{Im} \lambda_V \subset F^{\times}$. For example, if $\lambda_V$ is surjective, then $\Omega$ is independent of $\psi$. For any irreducible representation
$\pi$ of $\GO(V)$ (resp. $\GSp(W)^+$),   the maximal $\pi$-isotypic quotient of $\Omega$ has the form
\[    \pi \otimes  \Theta(\pi) \]
where $\Theta(\pi)$ is some smooth representation of $\GSp(W)^+$ (resp. $\GO(V)$). Further, we let
$\theta(\pi)$ be the maximal semisimple quotient of $\Theta(\pi)$. Note that though $\Theta(\pi)$
may be reducible, it has a central character $\omega_{\Theta(\pi)}$ given by
\[  \omega_{\Theta(\pi)} = \chi_V^{\frac{\dim W}{2}} \cdot \omega_{\pi}. \]
The extended Howe conjecture for similitudes says that $\theta(\pi)$ is irreducible whenever
$\Theta(\pi)$ is non-zero, and the map $\pi \mapsto \theta(\pi)$ is injective on its domain.
It was shown by Roberts \cite{Ro1} that this follows from the Howe conjecture for isometry groups, and thus holds if $p \ne 2$.  In any case, we have the following lemma which relates the theta correspondence for isometries and similitudes:
\vskip 5pt

\begin{Lem} \label{L:similitude1}

(i) Suppose that $\pi$ is an irreducible representation of a similitude group and $\tau$ is a constituent of the restriction of $\pi$ to the isometry group. Then $\theta_{\psi}(\tau) \ne 0$ implies that $\theta(\pi) \ne 0$.

\vskip 10pt

\noindent (ii) Suppose that
\[  \Hom_R(\Omega, \pi_1 \boxtimes \pi_2) \ne 0. \]
Suppose further that for each constituent $\tau_1$ in the restriction of $\pi_1$ to $\OO(V)$,
$\theta_{\psi}(\tau_1)$ is irreducible and
the map $\tau_1 \mapsto \theta_{\psi}(\tau_1)$ is injective on the set of irreducible constituents of
$\pi_1|_{\OO(V)}$. Then there is a uniquely determined bijection
\vskip 5pt
\[ f:  \{ \text{irreducible summands of $\pi_1|_{\OO(V)}$} \}  \longrightarrow
 \{ \text{irreducible summands of $\pi_2|_{\Sp(W)}$} \} \]
\vskip 10pt
\noindent such that for any irreducible summand $\tau_i$ in the restriction of $\pi_i$ to the relevant isometry group,
\vskip 3pt
\[  \tau _2 = f(\tau_1)  \Longleftrightarrow
\Hom_{\OO(V) \times \Sp(W)} (\omega_{\psi}, \tau_1  \boxtimes \tau_2) \ne 0 . \]
\vskip 5pt
\noindent One has the  analogous statement with the roles of $\OO(V)$ and $\Sp(W)$ exchanged.
 \vskip 10pt

\noindent (iii)  If $\pi$ is a representation of $\GO(V)$  (resp. $\GSp(W)^+$) and the restriction of $\pi$ to the relevant isometry group is $\oplus_i  \tau_i$, then as representations of $\Sp(W)$ (resp. $\OO(V)$),
\[  \Theta(\pi)  \cong  \bigoplus_i \Theta_{\psi}(\tau_i). \]
In particular, $\Theta(\pi)$ is admissible of finite length. Moreover,
if $\Theta_{\psi}(\tau_i) = \theta_{\psi}(\tau_i)$ for each  $i$, then
\[  \Theta(\pi) = \theta(\pi). \]
\end{Lem}

\begin{proof}
(i) Without loss of generality, suppose that $\pi$ is a representation of $\GSp(W)^+$.  As a representation of $\GSp(W)^+$,
\[  \Omega = ind_{\Sp(W)}^{\GSp(W)^+} \omega_{\psi}. \]
Hence the result follows by Frobenius reciprocity.

\vskip 5pt
(ii) This is [Ro1, Lemma 4.2], taking note of  the results of [AP], where it was shown that restrictions of irreducible representations from similitude groups to isometry groups are multiplicity-free.

 \vskip 5pt

(iii) By symmetry, let us  suppose that $\pi$ is a representation of $\GSp(W)^+$. Then we have the following sequence of $\OO(V)$-equivariant isomorphisms:
\begin{align}
\Theta(\pi)^* &\cong  \Hom_{\GSp(W)^+}(\Omega, \pi) \notag \\
&\cong \Hom_{\Sp(W)}(\omega_{\psi}, \pi|_{\Sp(W)}) \quad \text{(by Frobenius reciprocity)} \notag \\
&\cong \bigoplus_i \Hom_{\Sp(W)}(\omega_{\psi}, \tau_i) \notag \\
&\cong \bigoplus_i \Theta_{\psi}(\tau_i)^*, \notag \end{align}
where $\Theta(\pi)^*$ denotes the full linear dual of $\Theta(\pi)$.
 Thus, we have an $\OO(V)$-equivariant isomorphism of $\OO(V)$-smooth vectors
\[  \Theta(\pi)^{\vee} \cong \bigoplus_i \Theta_{\psi}(\tau_i)^{\vee}. \]
Note that since $\Theta(\pi)$ has a central character, the subspace of $\GO(V)$-smooth vectors in $\Theta(\pi)^*$ is the same as its subspace of $\OO(V)$-smooth vectors. In other words, the contragredient of $\Theta(\pi)$ as a representation of $\GO(V)$ is the same as  its contragredient as a representation of $\OO(V)$.  Using the fact that the
$\Theta_{\psi}(\tau_i)$'s are admissible of finite length, the above identity implies that $\Theta(\pi)$ is admissible of finite length.
The desired result then follows by taking contragredient. Moreover, if $\Theta_{\psi}(\tau_i)$ is semisimple  for each $i$,  then it is clear from the above that $\Theta(\pi)$ is semisimple as well.
\end{proof}
\vskip 10pt

\begin{Prop} \label{P:super}
Suppose that $\pi$ is a supercuspidal representation of $\GO(V)$ (resp. $\GSp(W)^+$). Then we have:
\vskip 5pt

\noindent (i) $\Theta(\pi)$ is either zero or is an irreducible representation of $\GSp(W)^+$ (resp. $\GO(V)$).
\vskip 5pt

\noindent (ii) If $\pi'$ is another supercuspidal representation such that $\Theta(\pi') = \Theta(\pi) \ne 0$, then $\pi' = \pi$.
\end{Prop}
\vskip 5pt

\begin{proof}
(i) Without loss of generality, suppose that $\pi$ is a supercuspidal representation of $\GO(V)$ and $\Theta(\pi)$ is nonzero. By Lemma \ref{L:similitude1}(iii), if $\pi|_{\OO(V)} = \bigoplus_i \tau_i$, we have:
\[ \Theta(\pi) = \theta(\pi) = \bigoplus_i \theta_{\psi}(\tau_i).\]
By Lemma \ref{L:similitude1}(ii) and Prop. \ref{P:Kudla},  we see that any irreducible constituent $\Pi$
of $\theta(\pi)$ satisfies:
\[  \Pi|_{\Sp(W)} = \bigoplus_i \theta_{\psi}(\tau_i).\]
Thus we see that $\Theta(\pi)$ is irreducible. This proves (i).
\vskip 5pt

(ii) Prop. \ref{P:Kudla}(ii) implies that if $\Theta(\pi') = \Theta(\pi) \ne 0$, then
$\pi'|_{\OO(V)}  \cong  \pi|_{\OO(V)}$. Since $\pi$ and $\pi'$ must have the same central character, we see that  $\pi' = \pi \otimes (\chi \circ \lambda_V)$ for some quadratic character $\chi$. Moreover, it is easy to see that
\[  \Theta(\pi \otimes (\chi\circ \lambda_V)) = \Theta(\pi) \otimes (\chi \circ \lambda_W). \]
Hence we would be done if we can show that for any quadratic character $\chi$,
\[  \pi \otimes \chi = \pi \Longleftrightarrow  \Theta(\pi) \otimes \chi = \Theta(\pi). \]
Of course, the implication $(\Longrightarrow)$ is clear from the above. To show the converse, let us set
\[  I(\pi) = \{ \text{quadratic characters $\chi$: $\pi \otimes \chi = \pi$} \}, \]
and let $I(\Theta(\pi))$ be the analogous group of quadratic characters. As we noted above,
\[ I(\pi) \subset I(\Theta(\pi)), \]
and we need to show the reverse inclusion.
Now the size of the group $I(\pi)$ is equal to the number of irreducible constituents in
$\pi|_{\OO(V)}$. By Lemma \ref{L:similitude1}(ii), however, the number of irreducible constituents in $\pi|_{\OO(V)}$ and $\Theta(\pi)|_{\Sp(W)}$ are equal. Hence
\[  I(\pi) = I(\Theta(\pi)), \]
as desired.

\end{proof}

\vskip 15pt

\section{\bf Theta Correspondences for $\GSp_4$} \label{S:dual}

In this section,  we specialize to the cases of interest in this paper.
Let $D$ be a (possibly split) quaternion algebra over $F$ and let $\mathbb{N}_D$ be its reduced norm.
Then $(D, \mathbb{N}_D)$ is a rank 4 quadratic space.
We have an isomorphism
\[  \GSO(D, \mathbb{N}_D) \cong (D^{\times} \times D^{\times})/ \{(z,z^{-1}): z \in \GL_1\} \]
via the action of the latter on $D$ given by
\[  (\alpha, \beta) \mapsto  \alpha x \overline{\beta}. \]
Moreover, an element of $\GO(D,\mathbb{N}_D)$ of determinant $-1$ is given by the conjugation action $c: x \mapsto \overline{x}$ on $D$.  An irreducible  representation of $\GSO(D)$ is thus of the form
$\tau_1 \boxtimes \tau_2$ where $\tau_i$ are representations of $D^{\times}$ with equal central character. Moreover, the action of $c$ on representations of $\GSO(D)$ is given by $\tau_1 \boxtimes \tau_2 \mapsto \tau_2 \boxtimes \tau_1$. One of the dual pairs which will be of interest to us is $\GSp_4 \times \GO(D)$. In the companion paper \cite{GT4}, we determine the associated theta correspondence completely. Some of the results are summarized in Theorem \ref{T:summary} below.

\vskip 10pt

Now consider the rank 6 quadratic space:
\[  (V_D, q_D)  = (D , \mathbb{N}_D) \oplus \mathbb{H} \]
where $\mathbb{H}$ is the hyperbolic plane. Then one has an isomorphism
\[  \GSO(V_D) \cong (\GL_2(D)  \times \GL_1)/ \{ (z \cdot \operatorname{Id},\; z^{-2}): z \in \GL_1\}. \]
To see this, note that the quadratic space $V_D$ can also be described as the space of
$2 \times 2$-Hermitian matrices with entries in $D$, so that a typical element has the form
\[  (a,d; x) = \left( \begin{array}{cc}
a & x \\
\overline{x} & d \end{array} \right),  \qquad \text{$a, d \in F$ and $x \in D$},  \]
equipped with the quadratic form $- \det (a,d;x) =  -ad + \mathbb{N}_D(x)$.
The action of $\GL_2(D)  \times \GL_1$ on this space is given by
\[  (g,z)(X) = z \cdot g \cdot X \cdot \overline{g}^t. \]
\noindent
Observe that an irreducible representation of $\GSO(V_D)$ is of the form $\pi \boxtimes \mu$ where $\pi$ is a representation of $\GL_2(D)$ and $\mu$ is a square root of the central character of $\pi$.  The similitude factor of $\GSO(V_D)$ is given by $\lambda_{D}(g,z) = N(g) \cdot z^2$,
where $N$ is the reduced norm on the central simple algebra $\text{M}_2(D)$. Thus,
 \[  \SO(V_D) = \{ (g,z) \in \GSO(V_D): N(g) \cdot z^2 = 1\}. \]

 \vskip 10pt
 We can now consider the theta correspondence in this case.
Since we only need to consider $V_D$ when $D$ is split,  we shall suppress $D$ from the notations.
Thus we specialize the results of the previous section to the case when $\dim W = 4$ and
$V$ is the split quadratic space of dimension $6$, so that $\lambda_V$ is surjective and
 the induced Weil representation $\Omega$ is a representation of $R = \GSp(W) \times \GO(V)$.
In fact, we shall only consider the
theta correspondence for $\GSp(W) \times \GSO(V)$. There is no significant loss in restricting to
$\GSO(V)$ because of the following lemma:
\vskip 5pt

\begin{Lem}
Let $\pi$ (resp. $\tau$) be an irreducible representation of $\GSp(W)$ (resp. $\GO(V)$) and
suppose that
\[  \Hom_{\GSp(W) \times \GO(V)}(\Omega, \pi \otimes \tau) \ne 0. \]
Then the restriction of $\tau$ to $\GSO(V)$ is irreducible. If $\nu_0 = \lambda_V^{-3}  \cdot \det$ is the unique non-trivial quadratic character of $\GO(V)/\GSO(V)$, then $\tau \otimes \nu_0$ does not participate in the theta correspondence with $\GSp(W)$.
\end{Lem}
 \vskip 5pt

 \begin{proof}
The analogous result for isometry groups is a well-known result of Rallis \cite[Appendix]{R}
 (see also \cite[\S 5, Pg. 282]{Pr1}). The lemma follows easily from this and we omit the details.
 \end{proof}

\vskip 10pt

 We now collect some results concerning the theta correspondence for  $\GSp(W) \times \GSO(V)$.
 Firstly,  we have:

 \begin{Thm}  \label{T:Howe}
 Let $\pi$ be an irreducible representation of $\GSp(W)$.  Then $\theta(\pi)$ is irreducible or zero as a representation of $\GSO(V)$. Moreover, if $\theta(\pi)  = \theta(\pi') \ne 0$, then $\pi = \pi'$.
  \end{Thm}

\begin{proof}
For supercuspidal representations, the result follows by  Prop. \ref{P:super} and the previous lemma. For non-supercuspidal representations, the result follows by the explicit determination of the theta correspondence for $\GSp(W) \times \GSO(V)$ given in the companion paper 
\cite[Thm.8.3]{GT4}.
\end{proof}
\vskip 5pt

Suppose now that $U$ (resp. $U_0$) is the unipotent radical of a Borel subgroup of $\GSp(W)$ (resp. $\GSO(V)$) and $\chi$ (resp. $\chi_0$) is a generic character of $U$ (resp. $U_0$). One may compute the twisted Jacquet module $\Omega_{U_0,\chi_0}$. The following lemma (see \cite[Prop. 7.4]{GT1} and \cite[Prop. 4.1]{MS}) describes the result:
\vskip 5pt

\begin{Lem} \label{L:Whit}
As a representation of $\GSp(W)$,
\[  \Omega_{U_0,\chi_0} = ind_U^{\GSp(W)} \chi. \]
In particular, if $\pi$ is an irreducible  generic representation of $\GSp(W)$, then $\Theta(\pi)$ is nonzero.
\end{Lem}

\vskip 10pt

Finally we describe the functoriality of the above
theta correspondence for unramified representations.
The $L$-group of $\GSp(W)$ is $\GSp_4(\C)$ and so an unramified representation of $\GSp(W)$ corresponds to a semisimple class in $\GSp_4(\C)$.
On the other hand, the $L$-group of $\GSO(V)$ is the subgroup of $\GL_4(\C) \times \GL_1(\C)$ given by
\[  {^L}\GSO(V) = \{(g,z) \in \GL_4(\C) \times \GL_1(\C): \det(g) = z^2 \}, \]
which is isomorphic to the group $\GSpin_6(\C)$. There is a natural map
\[  \iota:  {^L}\GSp_4  \longrightarrow  {^L}\GSO(V) \]
given by
\[  g \mapsto (g, \simi(g)) \]
where $\simi: \GSp_4(\C) \rightarrow \C^{\times}$ is the similitude factor. The following is shown in the companion paper \cite[Cor. 12.3]{GT4}:

 \begin{Prop}  \label{P:unram}
 Let $\pi = \pi(s)$ be an unramified representation of $\GSp(W)$ corresponding to the semisimple class $s \in \GSp_4(\C)$. Then $\theta(\pi(s))$ is the unramified representation of $\GSO(V)$ corresponding to the semisimple class $\iota(s) \in \GL_4(\C) \times \GL_1(\C)$.
\end{Prop}

\vskip 15pt

\section{\bf On Certain $L$-functions} \label{S:certain}

In this section, we introduce certain $L$-functions, $\epsilon$-factors and $\gamma$-factors which we will need. These local factors were defined by Shahidi \cite{Sh}. To specify them more precisely, we need to consider certain representations of the relevant $L$-groups.
\vskip 10pt

Recall that we have an inclusion of $L$-groups:
\[  \iota:  {^L}\GSp_4  \longrightarrow  {^L}\GSO(V) \cong \GSpin_6(\C) \subset \GL_4(\C) \times \GL_1(\C). \]
The projection of $\GL_4(\C) \times \GL_1(\C)$ onto the first factor thus defines a natural 4-dimensional representation of ${^L}\GSO(V)$ (one of the half-spin representations of $\GSpin_6(\C)$) whose restriction to ${^L}\GSp_4 \cong \GSp_4(\C)$ is the natural 4-dimensional representation of $\GSp_4(\C)$. Following a terminology common in the literature, we call this representation the Spin representation of ${^L}\GSp_4$ and ${^L}\GSO(V)$.
\vskip 5pt

Now, corresponding to the inclusion  $\SO(V) \hookrightarrow \GSO(V)$, one has a map of $L$-groups
\[  std: {^L}\GSO(V) \longrightarrow {^L}\SO(V) = \SO_6 (\C). \]
Indeed, one has the map
\[  \GL_4(\C) \times \GL_1(\C) \longrightarrow \GSO_6 (\C) \]
given by
\[  (g,z) \mapsto  z^{-1} \cdot \wedge^2 g, \]
and the map $std$ is simply the restriction of this map to the subgroup ${^L}\GSO(V) \subset \GL_4(\C) \times \GL_1(\C)$.
Similarly, corresponding to the inclusion $Sp(W) \hookrightarrow \GSp(W)$, we have a map of $L$-groups
\[  std: \GSp_4(\C) \longrightarrow \SO_5(\C) \cong \PGSp_4(\C) \]
and a commutative diagram
\[  \begin{CD}
\GSp_4(\C) @>\iota>> {^L}\GSO(V) @>>> \GL_4(\C) \times \GL_1(\C) \\
@VstdVV      @VVstdV    @VVstdV\\
\SO_5(\C)  @>\iota_0>>  \SO_6 (\C) @>>>\qquad \GSO_6 (\C)\qquad .\end{CD} \]
We regard $std$ as a 5-dimensional (resp. 6-dimensional) representation of ${^L}\GSp_4$ (resp. ${^L}\GSO(V))$ and call it the standard representation.
\vskip 5pt

Thus, for a representation $\pi$ of $\GSp_4(F)$, one expects to  be able to define a
standard degree 5 $L$-function $L(s, \pi, std)$ and a degree 4 Spin $L$-function $L(s, \pi, spin)$.
More generally, for representations $\pi$ of $\GSp_4(F)$ and $\sigma$ of $\GL_r(F)$, one expects to have the $L$-functions
\[  L(s, \pi \times \sigma, std \boxtimes std) \quad \text{and} \quad L(s, \pi \times \sigma, spin \boxtimes std), \]
which are associated to the representations $std \boxtimes std$ and $spin \boxtimes std$ respectively.
\vskip 5pt

Similarly, given a representation $\Sigma$ of $\GSO(V)$, one expects to have the degree 6 $L$-function $L(s, \Sigma, std)$ and the degree 4 $L$-function $L(s,  \Sigma, spin)$. More generally, for representations $\Sigma$ of $\GSO(V)$ and $\sigma$ of $\GL_r(F)$, one expects to have the $L$-functions
\[ L(s, \Sigma \times \sigma , std \boxtimes std) \quad \text{and} \quad  L(s, \Sigma \times \sigma, spin \boxtimes std). \]
Moreover, if we regard $\Sigma$ as a representation $\Pi \boxtimes \mu$ of $\GL_4(F) \times \GL_1(F)$, then $L(s,\Sigma \times \sigma,spin \boxtimes std)$ should be nothing but the Rankin-Selberg $L$-function $L(s, \Pi \times \sigma)$ of the representation $\Pi \boxtimes \sigma$ of $\GL_4(F) \times \GL_r(F)$. Further, one expects that
\[ L(s, \Sigma, std) =  L(s, \Pi, {\bigwedge}^2 \otimes \mu^{-1}), \]
where the $L$-function on the RHS is the twisted exterior square $L$-function.
\vskip 5pt

In the important paper \cite{Sh}, the above $L$-functions and their associated $\epsilon$-factors were defined by Shahidi when the representation $\pi \boxtimes \sigma$ or $\Sigma$ is generic. More precisely, suppose that
\vskip 5pt
\begin{itemize}
\item $M \subset G$ is the Levi subgroup of a maximal parabolic subgroup $P = M \cdot N$;
\vskip 5pt

\item $\tau$ is an irreducible generic representation of $M(F)$;
\vskip 5pt

\item the adjoint action of the dual group $M^{\vee}$ on $\mathfrak{n}^{\vee} = Lie(N^{\vee})$ decomposes as $r_1 \oplus r_2 \oplus...\oplus r_k$, where each $r_i$ is a maximal  isotypic component for the action of the central torus in $M^{\vee}$.
\end{itemize}
\vskip 5pt

\noindent Then Shahidi defined the local factors $\gamma(s, \tau, r_i, \psi)$, $L(s, \tau, r_i)$, $\epsilon(s, \tau, r_i,\psi)$ which satisfy
\[  \gamma(s, \tau,r_i,\psi) = \epsilon(s, \tau, r_i,\psi) \cdot \frac{L(1-s, \tau^{\vee}, r_i)}{L(s,\tau, r_i)}. \]
\vskip 5pt

In Table \ref{table}, we list the various $L$-functions that we will use, and the data $(M,G, \tau)$ which are used in their definition via the Shahidi machinery.

\begin{table}
\centering \caption{On Certain $L$-functions}
\label{table} \vspace{-3ex}
$$
\renewcommand{\arraystretch}{1.5}
 \begin{array}{|c|c|c|c|c|c|}
  \hline
&  \mbox{$L$-function} & \mbox{M}
   & \mbox{G} & \mbox{$r_i$} & \tau  \\\hline\hline
  \mbox{a} & L(s, \pi \times \sigma, spin\boxtimes std) &  &
   & r_1 = spin^{\vee} \boxtimes std &  \\  \cline{1-2} \cline{5-5}
\mbox{b} & L(s, \sigma, Sym^2 \otimes \omega_{\pi}) & \raisebox{2.5ex}[-2.5ex]{\mbox{$\GSpin_5 \times \GL_r$}} &  \raisebox{2.5ex}[-2.5ex]{\mbox{$\GSpin_{2r+5}$}} & r_2 = sim^{-1} \otimes (Sym^2 std) &
  \raisebox{2.5ex}[-2.5ex]{\mbox{$\pi^{\vee} \boxtimes \sigma$}} \\ \hline
\mbox{c} &   L(s, \pi \times \sigma, std \boxtimes std) &   \GSp_4 \times \GL_r   & \GSp_{2r+4}
   &    r_1 = std^{\vee} \boxtimes std & \pi^{\vee} \boxtimes \sigma  \\ \hline
\mbox{d} & L(s, \Sigma , std)  & \GSO_6  \times \GL_1 & \GSO_8 & r_1 = std^{\vee} \boxtimes std &
\Sigma^{\vee}  \boxtimes {\bf 1}  \\ \hline
\mbox{e}  & L(s, \Pi \times \sigma)  & \GL_4 \times \GL_r & \GL_{r+4}
   &  r_1 = std \boxtimes std^{\vee} & \Pi \boxtimes \sigma^{\vee} \\ \hline
\mbox{f}  & L(s, \Pi, \bigwedge^2 \otimes \mu^{-1}) & \GL_4 \times \GL_1 & \GSpin_8
   &  r_2 = (\bigwedge^2 std) \boxtimes std^{\vee} & \Pi \boxtimes \mu \\ \hline
  \end{array}
$$
\end{table}

\vskip 10pt

Though Shahidi defined the $L$-functions in Table \ref{table} only for generic representations, the definition can be extended to non-generic non-supercuspidal representations of the simple factors of the groups $M$ which occur in the table.  This uses the Langlands classification, which says that every  irreducible admissible representation can be expressed as the unique quotient of a standard module, i.e. one induced from a non-negative twist of a discrete series representation of a Levi subgroup.
For the groups $M$ occurring in Table \ref{table}, their simple factors have proper Levi subgroups which  are essentially products of $\GL_k$'s, so that a discrete series  representation of such a proper Levi subgroup is generic. Moreover, the restriction of each representation $r_i$ to a proper Levi subgroup of $M$ decomposes into the sum of irreducible constituents, all of which appears in the setup of Shahidi's theory. Thus, one may extend the definition of  the local factors to all non-generic non-supercuspidal representations of each simple factor of $M$ by multiplicativity (with respect to the standard module under consideration). Thus, the local factors given in Table \ref{table} are defined except when the representation $\pi$ of $\GSp_4(F) \cong \GSpin_5(F)$ is non-generic supercuspidal.
\vskip 10pt

The $L$-function $L(s, \pi \times \sigma, spin\boxtimes std)$ in (a) of Table \ref{table} is the main one which intervenes in our main theorem. Hence,  we shall simplify notations by writing it as $L(s, \pi \times \sigma)$, suppressing the mention of $spin \boxtimes std$. The same comment applies to the
$\epsilon$- and $\gamma$-factors.
\vskip 10pt

Finally, we note the following two lemmas:
\vskip 5pt

\begin{Lem} \label{L:4.1}
Let $\Sigma$ be an irreducible generic representation of $\GSO(V)$ which we may identify with a representation $\Pi \boxtimes \mu$ of $\GL_4(F) \times \GL_1(F)$ via the isomorphism
\[ \GSO(V) \cong \GL_4(F) \times \GL_1(F)/ \{(t,t^{-2}): t \in F^{\times}\}. \]
 Then we have:
\[  L(s, \Sigma, std) = L(s, \Pi, {\bigwedge}^2 \otimes \mu^{-1}), \]
where the $L$-function on the LHS (resp. RHS) is that in (d) (resp. (f)) of the above table. Moreover, one has the analogous identity for the $\epsilon$- and $\gamma$-factors.
\end{Lem}

\vskip 10pt

\begin{Lem} \label{L:4.2}
Suppose that $\Sigma$ is an irreducible generic representation of a similitude group $\GSp(W)$ or $\GSO(V)$ and $\Sigma_0$ is
an irreducible constituent of the restriction of $\Sigma$ to the isometry group $\Sp(W)$ or $\SO(V)$.  Then one has:
\[  L(s, \Sigma, std) = L(s, \Sigma_0, std)\quad \text{and} \quad
\epsilon(s,\Sigma, std, \psi) = \epsilon(s,\Sigma_0, std, \psi). \]
 \end{Lem}
\vskip 5pt

\noindent Both of these lemmas follow from the characterization of Shahidi's local factors given in \cite[Thm. 3.5]{Sh}.

 \vskip 15pt

\section{\bf The Results of Kudla-Rallis, Mui\'{c}-Savin and Henniart}  \label{S:KRMS}\
In this section, we review some crucial general results of Kudla-Rallis \cite{KR}, Mui\'{c}-Savin \cite{MS} and Henniart \cite{He2} before specializing them to the cases of interest in this paper.
\vskip 5pt

Let $W_n$ be the 2n-dimensional symplectic vector space with associated
symplectic group $\Sp(W_n)$ and consider the two towers of orthogonal groups attached to the quadratic spaces with trivial discriminant. More precisely, let
\[  V_m = \mathbb{H}^m \quad \text{and} \quad V^{\#}_m = D \oplus \mathbb{H}^{m-2} \]
and denote the orthogonal groups by $\OO(V_m)$ and $\OO(V^{\#}_{m})$ respectively.
For an irreducible representation $\pi$ of $\Sp(W_n)$, one may consider the theta lifts
$\theta_m(\pi)$ and $\theta^{\#}_m(\pi)$ to $\OO(V_m)$ and $\OO(V^{\#}_{m})$ respectively (with respect to
a fixed non-trivial additive character $\psi$). Set
\[  \begin{cases}
m(\pi) = \inf \{ m: \theta_m(\pi) \ne 0 \}; \\
m^{\#}(\pi) = \inf \{m: \theta^{\#}_m(\pi) \ne 0 \}.
\end{cases} \]
\vskip 5pt

\noindent Then Kudla and Rallis  \cite[Thms. 3.8 \& 3.9]{KR} showed:
\vskip 5pt

\begin{Thm}  \label{T:KR}
(i)  For any irreducible representation $\pi$ of $\Sp(W_n)$,
\[  m(\pi) + m^{\#}(\pi) \geq 2n  +2. \]
\vskip 5pt

\noindent (ii) If $\pi$ is a supercuspidal representation of $\Sp(W_n)$, then
\[  m(\pi) + m^{\#}(\pi) = 2n  +2. \]
\end{Thm}
\vskip 5pt

If we specialize this result to the case $\dim W_n = 4$ and take into account the results of the companion paper \cite{GT4}, we obtain:
 \vskip 5pt

\begin{Thm} \label{T:dichotomy}
Let $\pi$ be an irreducible representation of $\GSp_4(F)$. Then one has the following two mutually exclusive possibilities:
\vskip 5pt

\noindent (A) $\pi$ participates in the theta correspondence with $\GSO(D) =\GSO_{4,0}(F)$, where $D$ is non-split;
\vskip 5pt

\noindent (B) $\pi$ participates in the theta correspondence with $\GSO(V) = \GSO_{3,3}(F)$.
\vskip 5pt

Another way of describing this result is that one of the following two possibilities holds:
\vskip 5pt

\noindent (I) $\pi$ participates in the theta correspondence with either $\GSO(D)$ or $\GSO(V_2)=\GSO_{2,2}(F)  $ (but necessarily not both);
\vskip 5pt

\noindent (II) $\pi$ does not participate in the theta correspondence with $\GSO(D)$ or $\GSO(V_2) $, in which case it must participate in the theta correspondence with $\GSO(V)$.
 \end{Thm}

\vskip 10pt
\begin{proof}
Theorem \ref{T:KR}(i) implies that any representation $\pi$ participates in the theta correspondence with at most one of $\GSO(D)$ or $\GSO(V)$. Hence it remains to show that any $\pi$ does participate in the theta correspondence with $\GSO(D)$ or $\GSO(V)$.
If $\pi$ is supercuspidal, this is an immediate consequence of Theorem \ref{T:KR}(ii).
For generic representations, it follows by Lemma \ref{L:Whit} that $\pi$ has nonzero theta lift to
$\GSO(V)$; in particular, this implies the theorem for essentially discrete series representations which are not supercuspidal, since these are generic.
For the remaining non-generic representations, the result follows by an explicit determination of theta correspondences for $\GSp_4$ given in the companion paper \cite{GT4}, especially \cite[Thms. 8.1 and 8.3]{GT4}.
\end{proof}

\vskip 5pt

Now we come to the results of Mui\'{c}-Savin \cite{MS}. In the setting above, they considered a discrete series representation of $\Sp(W_n)$ which is generic with respect to a character $\chi$ and determine the value of $m(\pi)$. Similarly, starting with a discrete series representation $\tau$ of $\SO(V_m)$, one may  define
$n(\tau)$ analogously. Here is the result of Mui\'{c}-Savin:

\vskip 5pt

\begin{Thm}  \label{T:MS}
(i) Suppose that $\pi$ is a discrete series representation of $\Sp(W_n)$ which is generic with respect to a character $\chi$.

\begin{itemize}
\item[(a)]  If the standard $L$-factor $L(s, \pi, std)$ has a pole at $s = 0$, then $m(\pi) = n$ and all the irreducible constituents (as $\SO(V_{m(\pi)})$-modules) of $\theta_{m(\pi)}(\pi)$ are discrete series representations which are generic with respect to $\chi$.

\item[(b)] If the standard $L$-factor $L(s, \pi, std)$ does not have a pole at $s = 0$, then  $m(\pi) = n+1$ and $\theta_{m(\pi)}(\pi)$ has a unique $\chi$-generic constituent.  Moreover, this $\chi$-generic constituent is a discrete series representation.
\end{itemize}
\vskip 5pt

(ii) Suppose that $\tau$ is a discrete series representation of $\SO(V_m)$ which is generic with respect to a character $\chi$.

\begin{itemize}
\item[(a)]  If the standard $L$-factor $L(s, \tau, std)$ has  a pole at $s = 0$, then $n(\tau) = m-1$ and  all the irreducible constituents of $\theta_{n(\tau)}(\tau)$ are discrete series representations which are generic with respect to $\chi$.

\item[(b)]  If the standard $L$-factor $L(s, \tau, std)$ does not have a pole at $s = 0$ , then $n(\tau) = m$ and  $\theta_{n(\tau)}(\tau)$ has a unique $\chi$-generic $\SO(V_{m(\pi)})$-constituent. Moreover, this $\chi$-generic constituent is a discrete series representation.
\end{itemize}
\end{Thm}
\vskip 10pt

\begin{Cor} \label{C:MS}
Let $\Sigma = \Pi \boxtimes \mu$ be a (necessarily generic)  discrete series representation of $\GSO(V)$ and suppose that
the twisted exterior square $L$-function $L(s, \Pi, \bigwedge^2  \otimes \mu^{-1})$ has a pole at $s = 0$.
Then the theta lift of $\Sigma = \Pi \boxtimes \mu$ to $\GSp_4(F)$ is a nonzero generic discrete series representation.
\end{Cor}

\begin{proof}
Let $\Sigma_0$ be an irreducible constituent of the restriction of $\Sigma = \Pi \boxtimes \mu$ to $\SO(V)$, so that $\Sigma_0$ is $\chi$-generic with respect to some $\chi$.
By Lemmas \ref{L:4.1} and \ref{L:4.2},
the standard $L$-function of $\Sigma_0$ is
\[  L(s, \Sigma_0, std) = L(s, \Pi, \bigwedge^2 \otimes \mu^{-1}) \]
and thus has a pole at $s = 0$. By Thm. \ref{T:MS}(ii)(a), the theta lift $\theta(\Sigma)$ of $\Sigma$ to $\GSp_4(F)$ is nonzero and all its constituents are $\chi$-generic discrete series representations. This shows that
$\theta(\Pi \boxtimes \mu)$ is a nonzero generic discrete series representation, which is irreducible by Thm. \ref{T:Howe}.
\end{proof}
\vskip 10pt

Finally, in our application of Cor. \ref{C:MS} later on, we need the following crucial result of Henniart \cite{He2}:
\vskip 5pt

\begin{Thm} \label{T:henniart}
The local Langlands correspondence for $\GL_{n}$ respects the twisted exterior square $L$-function.
In other words, if $\Pi$ is an irreducible representation of $\GL_{n}(F)$ with $L$-parameter $\phi_{\Pi}$ and $\mu$ is a 1-dimensional character of $F^{\times}$, then
\[  L(s, \Pi, {\bigwedge}^2 \otimes \mu^{-1}) = L(s, {\bigwedge}^2 \phi_{\Pi} \otimes \mu^{-1}), \]
where the $L$-function on the LHS is that of Shahidi.
\end{Thm}

\vskip 10pt

The dichotomy given in Theorem \ref{T:dichotomy} is qualitative in nature, but can be made more concrete by the explicit determination of the three theta correspondences in question. This detailed study of theta correspondences is given in the companion paper \cite{GT4}.  We conclude this section by assembling the above results together with those of \cite{GT4}.
\vskip 5pt

\begin{Thm}  \label{T:summary}
(i) The set of irreducible representations of $\GSp_4(F)$ which are of Type (A) is precisely
 \[ \Pi(\GSp_4)_{ng}^{temp}: =
 \{ \text{non-generic essentially tempered representations of $\GSp_4(F)$} \}. \]
More precisely,  under the theta correspondence for $\GSO(D) \times \GSp_4$,
the map
\[  \tau^D_1 \boxtimes \tau^D_2  \mapsto \theta(\tau^D_1 \boxtimes \tau_2^D) \]
 defines a bijection
 \[  \text{$\Pi(\GSO(D))$ modulo action of $\GO(D)$} \longleftrightarrow \Pi(\GSp_4)_{ng}^{temp}. \]
Moreover, the image of the subset of $\tau^D_1 \boxtimes \tau^D_2$'s,  with $\tau^D_1 \ne \tau^D_2$, is precisely the subset of non-generic supercuspidal representations of $\GSp_4(F)$. The other representations in the image are the non-discrete series representations in \cite[Table 1, NDS(c)]{GT4}.
\vskip 5pt

\noindent (ii) The theta correspondence for $\GSO(V_2) \times \GSp_4$ defines an injection
\[  \text{$\Pi(\GSO(V_2))$ modulo action of $\GO(V_2)$} \longrightarrow \Pi(\GSp_4). \]
The image is disjoint from $\Pi(\GSp_4)_{ng}^{temp}$ and consists of:
\begin{itemize}
 \item[(a)]  the generic discrete series representations (including supercuspidal ones) such that $L(s, \pi, std)$ has a pole at $s = 0$.

\item[(b)] the non-discrete series representations in \cite[Table 1, NDS(b, d,e)]{GT4}.
\end{itemize}
 Moreover, the images of the representations $\tau_1 \boxtimes \tau_2$'s, with $\tau_1 \ne \tau_2$ discrete series representations of $\GL_2(F)$, are precisely the representations in (a).
\vskip 5pt

\noindent (iii) The theta correspondence for $\GSp_4 \times \GSO(V)$ defines an injection
\[  \Pi(\GSp_4) \smallsetminus \Pi(\GSp_4)_{ng}^{temp} \longrightarrow \Pi(\GSO(V)) \subset \Pi(\GL_4) \times \Pi(\GL_1). \]
Moreover, the representations  of $\GSp_4(F)$ which are of Type (II), i.e. those not accounted for by (i) and (ii) above, are
\begin{itemize}
\item[(a)]  the generic discrete series representations $\pi$ whose standard factor $L(s,\pi, std)$ is holomorphic at $s=0$.   The images of these representations  under the above map are precisely the discrete series representations $\Pi \boxtimes \mu$ of
$\GL_4(F) \times  \GL_1(F)$ such that $L(s, {\bigwedge}^2 \phi_{\Pi} \otimes \mu^{-1})$ has a pole at $s = 0$.

\vskip 5pt
\item[(b)] the non-discrete series representations in \cite[Table 1, NDS(a)]{GT4}. The images of these under the above map consists of non-discrete series representations  $\Pi \boxtimes \mu$ such that
\[ \phi_{\Pi} = \rho \oplus \rho \cdot \chi \quad \text{and} \quad \mu = \det \rho \cdot \chi,  \]
for  an irreducible two dimensional $\rho$ and a character $\chi \ne 1$.
\end{itemize}
\vskip 5pt

\noindent (iv) If a representation $\pi$ of $\GSp_4(F)$ with central character $\mu$ participates in the theta correspondence with  $\GSO(V_2)  $, so that
\[  \pi = \theta(\tau_1 \boxtimes \tau_2) = \theta(\tau_2 \boxtimes \tau_1),\]
then $\pi$ has a nonzero theta lift to $\GSO(V)$.
If $\Pi \boxtimes \mu$ is the small theta lift of $\pi$ to $\GSO(V)$, with $\Pi$ a representation of
$\GL_4(F)$, then
\[  \phi_{\Pi} = \phi_{\tau_1} \oplus \phi_{\tau_2} \quad \text{and} \quad  \mu = \det \phi_{\tau_1} = \det \phi_{\tau_2}. \]
  \end{Thm}

\vskip 5pt
\noindent The proof of Theorem \ref{T:summary} is given in [GT4], where more complete and explicit information can be found.

\vskip 10pt
\section{\bf Langlands Parameters}  \label{S:parameters}

In this section, we record some facts about the Langlands parameters for $\GSp_4$.
Given such a parameter
\[  \phi: WD_F \longrightarrow \GSp_4(\C), \]
we may consider its composition with the similitude factor $\simi$ to obtain a  1-dimensional character $\simi(\phi)$ of $WD_F$. We call $\simi(\phi)$ the similitude character of $\phi$.
\vskip 5pt

Now consider the composite of $\phi$ with the inclusion
\[  \iota: \GSp_4(\C) \hookrightarrow \GL_4(\C) \times \GL_1(\C) \]
to obtain an $L$-parameter for $\GL_4 \times \GL_1$. We note the following crucial lemma (cf. also [V,  \S 7]), which is the analog of Thm. \ref{T:summary}(iii) for $L$-parameters.
\vskip 5pt

\begin{Lem}  \label{L:par}
The map $\phi \mapsto \iota \circ \phi$ gives an injection
\[  \Phi(\GSp_4) \hookrightarrow \Phi(\GL_4) \times \Phi(\GL_1). \]
 The discrete series $L$-parameters $\phi \times \mu$ of $\GL_4 \times \GL_1$ which are in the image of the map are precisely those such that $L(s, \bigwedge^2 \phi \otimes \mu^{-1})$ has a pole at $s = 0$.
\end{Lem}

\begin{proof}
We shall prove the lemma in the context of  $\GSp_{2n}$. For any $n \geq 1$, we shall show that the natural map
\[  \Phi(\GSp_{2n}) \longrightarrow \Phi(\GL_{2n}) \times \Phi(\GL_1) \]
is injective. This injectivity is equivalent to the following assertion: if $\phi: WD_F \longrightarrow \GL(V)$ is an $L$-parameter, with $V$ a $2n$-dimensional vector space over $\C$,
and $B$ and $B'$ are two nondegenerate symplectic forms on $V$ which are preserved by $\phi$ up to a similitude character $\mu$, then $B$ and $B'$ are conjugate by an element of $\GL(V)$ which centralizes $\phi$.
\vskip 5pt

We now  prove the above statement. Since $(\phi,V)$ is semisimple, we may write
\[  V = \bigoplus_i M_i \otimes V_i \]
where $V_i$ is irreducible and $M_i$ is its multiplicity space.  Since $V^{\vee}  \otimes \mu \cong V $, we see that, for each $i$, either $V_i^{\vee}\otimes \mu  \cong V_i $, or else $V_i^{\vee}\otimes \mu  = V_j $ for some $j \ne i$. Corresponding to these two possibilities,  we may decompose $V$ as:
\[  V = \left( \bigoplus_i M_i \otimes V_i \right) \oplus  \left( \bigoplus_j M_j \otimes (W_j \oplus W_j^{\vee} \cdot \mu) \right). \]
Since the nondegenerate forms $B$ and $B'$ remain nondegenerate on each summand above, we are reduced to showing the statement on each summand.
\vskip 5pt

First examine $M_i \otimes V_i$. Since $V_i$ is irreducible, it follows by Schur's lemma that there is a unique (up to scaling) bilinear form on $V_i$ which is $WD_F$-equivariant with similitude character $\mu$. Any such nonzero form is necessarily nondegenerate and has a sign $\epsilon_i$. Thus, giving a nondegenerate $(WD_F ,\mu)$-equivariant symplectic form on $M_i \otimes V_i$ is equivalent to giving a nondegenerate bilinear form on $M_i$ of sign $-\epsilon_i$.  But any two such forms on $M_i$ are conjugate under $\GL(M_i)$, which commutes with the action of $WD_F$. This proves the statement for the summand $M_i \otimes V_i$.
 \vskip 5pt

Now examine the summand $W = M_j \otimes (W_j \oplus W_j^{\vee} \cdot \mu)$. In this case, the subspaces $M_j \otimes W_j$ and $M_j \otimes W_j^{\vee}\cdot \mu$ are necessarily totally isotropic (with respect to any $(WD_F, \mu)$-equivariant symplectic form). Moreover, there is a unique $(WD_F, \mu)$-equivariant pairing on $W_j \times W_j^{\vee} \cdot \mu$, up to scaling. Hence, giving
 a nondegenerate $(WD_F,\mu)$-equivariant symplectic form on  $W$ is equivalent to giving a nondegenerate bilinear form on $M_j$. But any two such are conjugate under the natural action of $GL(M_j) \times GL(M_j)$, which commutes with the action of $WD_F$. This proves the statement for the summand $W$.

 \vskip 5pt
 Finally,  suppose that $\phi \times \mu \in \Phi(\GL_{2n}) \times \Phi(\GL_1)$ is such that
$\phi$ is irreducible as a $2n$-dimensional representation. Then it is not difficult to see that
$L(s,\bigwedge^2 \phi \otimes \mu^{-1})$  has a pole at $s = 0$ if and only if $\left( \bigwedge^2 \phi \right) \otimes \mu^{-1}$ contains the  trivial representation as a summand. In other words, the action of $WD_F$ via $\phi$ preserves a non-zero symplectic form up to scaling by the character $\mu$. This symplectic form is necessarily nondegenerate, so that
$\phi$ factors through $\GSp_{2n}(\C)$ after conjugation, for otherwise its kernel would be a non-trivial submodule.

\end{proof}
\vskip 5pt

The theory of endoscopy shows that $\GSp_4$ has a unique endoscopic group which is isomorphic to $\GSO_{2,2}$. The dual group of $\GSO_{2,2}$ is
\[  \GSpin_4(\C) \cong (\GL_2(\C) \times \GL_2(\C))^0 = \{ (g_1,g_2): \det g_1 = \det g_2\}, \]
so that there is a distinguished conjugacy class of embeddings of dual groups
\[  (\GL_2(\C) \times \GL_2(\C))^0 \hookrightarrow \GSp_4(\C). \]
This gives rise to a natural map
\[  \Phi(\GSO_{2,2}) \longrightarrow \Phi(\GSp_4). \]
We say that an $L$-parameter $\phi  \in \Phi(\GSp_4)$ is  {\bf endoscopic} if it is in the image of this map. More concretely, $\phi$ is endoscopic if
\[  \phi = \phi_1 \oplus \phi_2 \quad \text{with $\dim \phi_i =2$ and $\simi \phi = \det \phi_1= \det \phi_2$.} \]
Note that the $\phi_i$'s are not necessarily distinct and may be reducible.  Observe further that the outer automorphism group $\operatorname{Out}(\SO_4) \cong \Z/2\Z$ of $\SO_4$ acts on $\Phi(\GSO_{2,2})$ via $(\phi_1, \phi_2) \mapsto (\phi_2, \phi_1)$. It is clear that the natural map above descends to give
\[    \text{$\Phi(\GSO_{2,2})$ modulo action of $\operatorname{Out}(\SO_4)$} \longrightarrow \Phi(\GSp_4). \]

 \vskip 5pt
The following lemma is the analog of Thm. \ref{T:summary}(i) and  (ii) for $L$-parameters.

 \begin{Lem} \label{L:endoscopic}
 (i) The map
 \[  \text{$\Phi(\GSO_{2,2})$ modulo action of $\operatorname{Out}(\SO_4)$} \longrightarrow \Phi(\GSp_4) \]
 is injective.
 \vskip 5pt

\noindent  (ii) If $\phi$ is a discrete series parameter of $\GSp_4$, then $\phi$ is either endoscopic or is irreducible as a 4-dimensional representation. If $\phi$ is endoscopic, then $\phi = \phi_1 \oplus \phi_2$ with $\phi_1 \ncong \phi_2$ irreducible, and the component group $A_{\phi}$ is $\Z/2\Z$. Otherwise, $A_{\phi}$ is trivial.
 \vskip 5pt

 (iii) If $\phi$ is a non-discrete series parameter of $\GSp_4$, then $\phi$ is either endoscopic or
 $\phi = \rho \oplus \rho \cdot \chi$ with $\simi \phi = \det \rho \cdot \chi$ and $\chi \ne 1$.
 The component group $A_{\phi}$ is trivial unless $\phi = \phi_1 \oplus \phi_2$ is endoscopic with $\phi_1 \cong \phi_2$ irreducible, in which case $A_{\phi} = \Z/2\Z$.
 \vskip 5pt

 In particular, under the injection of Lemma \ref{L:par}, the non-endoscopic $L$-parameters consists precisely of
 those pairs $(\phi, \mu) \in \Phi(\GL_4) \times \Phi(\GL_1)$ which arise as $(\phi_{\Pi}, \mu)$ in Thm. \ref{T:summary}(iii)(a) and (b).
 \end{Lem}

 \begin{proof}
 \noindent (i)   By Lemma \ref{L:par}, it suffices to show that the natural map
 \[  \text{$\Phi(\GSO_{2,2})$ modulo action of $\operatorname{Out}(\SO_4)$} \longrightarrow \Phi(\GL_4) \times \Phi(\GL_1) \]
 is injective. This is a simple exercise which we leave to the reader.
  \vskip 5pt

 \noindent (ii) If $\phi$ is irreducible, then  the centralizer in $\GSp_4(\C)$ of the image of $\phi$
 is $Z_{\GSp_4}(Im \phi) = Z_{\GSp_4}$, so that $A_{\phi}$ is trivial.
 If $\phi$ is reducible, then the hypothesis that $\phi$ does not factor through any proper parabolic subgroup implies that $\phi$ does not stabilize any nonzero isotropic subspaces. Thus
 $\phi$ cannot contain any
1-dimensional summand and any 2-dimensional summand must be a nondegenerate symplectic subspace. This shows that $\phi = \phi_1 \oplus \phi_2$ with $\phi_i$ irreducible 2-dimensional and $\simi \phi = \det \phi_i$. Moreover, $\phi_1 \ncong \phi_2$, for otherwise $\phi$ would stabilize a 2-dimensional isotropic subspace. Thus, $\phi$ has the form given in (i). Moreover,
 \[  Z_{\GSp_4}(Im(\phi)) \cong  \{ (a,b) \in \C^{\times} \times \C^{\times} : a^2 = b^2 \}  \subset (\GL_2(\C) \times \GL_2(\C))^0 \]
with $Z_{\GSp_4}$ sitting diagonally as a subgroup. Thus, $A_{\phi} \cong \Z/ 2\Z$,
\vskip 5pt

\noindent (iii) If $\phi$ stabilizes an isotropic line $L$ which affords the character $\chi$, then $\phi$ must stabilize another line $L'$ which has nonzero pairing with $L$ and affords the character $\chi^{-1} \cdot \simi \phi$. The subspace spanned by $L$ and $L'$ supports a 2-dimensional submodule $\phi_1$ of $\phi$ with $\simi \phi = \det \phi_1$.  It follows that $\phi$ is endoscopic.
On the other hand, if $\phi$ stabilizes an isotropic plane, but not a line, then
\[ \phi  = \rho \oplus  \rho^{\vee} \cdot \simi \phi = \rho \oplus \rho \cdot \chi \]
for an irreducible 2-dimensional $\rho$ and with $\simi \phi  = \chi \cdot \det \rho$. If $\chi = 1$, then
$\phi = 2 \cdot \rho$ is endoscopic. If $\chi \ne 1$, then $\phi$ is non-endoscopic of the type given in (ii). We leave the determination of $A_{\phi}$ in the various cases to the reader.
 \end{proof}

\vskip 10pt

\noindent{\bf Remarks:}  In a letter [Z] to Vigneras (dated Nov. 25, 1984 and mentioned at the end of the introduction of [V]), W. Zink gave an argument that there do not exist primitive symplectic representations of $W_F$. However, this is not the case, as one can construct such L-parameters
when the residue characteristic $p$ of $F$ is $2$; see [GT2, Prop. 5.3].
\vskip 10pt

\section{\bf Construction of $L$-Packets and Exhaustion}  \label{S:packets}

In this section, we shall construct the map $L$, show that it is surjective and verify (i), (ii), (iii), (iv) and (vii) of the Main Theorem.
Since we are working with reductive but non-semisimple groups $G$, let us mention that for the rest of the paper, by a discrete series (resp. tempered) representation of $G(F)$, we mean a representation which is equal to a unitary discrete series (resp. tempered) representation after twisting by a 1-dimensional character.
\vskip 15pt

\noindent{\bf \underline{Definition of the Map $L$}}
\vskip 5pt

 According to Thm. \ref{T:dichotomy}, the irreducible representations $\pi$ of $\GSp_4(F)$ fall into two disjoint families of Type (I) or (II).
 \vskip 5pt

\noindent{\bf Type I:}  If $\pi$ is of Type (I), then there is an irreducible representation $\tau^D_1 \boxtimes \tau^D_2$
of $\GSO(D)$ (where $D$ is possibly split) such that
\[  \pi = \theta(\tau_1^D \boxtimes \tau_2^D) = \theta(\tau_2^D \boxtimes \tau_1^D). \]
By the Jacquet-Langlands correspondence and the local Langlands correspondence for $\GL_2$,
each $\tau^D_i$ gives rise to an irreducible 2-dimensional representation $\phi_i$ of $WD_F$
such that $\det \phi_1 = \det \phi_2$. We define $L(\pi)$ to be the parameter
\[  \phi  = \phi_1 \oplus \phi_2: WD_F \longrightarrow (\GL_2(\C) \times \GL_2(\C))^0 \subset \GSp_4(\C). \]
 By Thm. \ref{T:summary}(i, ii) and Lemma \ref{L:endoscopic}, we see that $L(\pi)$ is a discrete series parameter iff $\pi$ is a discrete series representation.

\vskip 10pt

\noindent{\bf Type II:}   If $\pi$ is of type (II), then  the theta lift of $\pi$ to $\GSO(V)$ is nonzero.  Regarding $\GSO(V)$ as a quotient of $\GL_4(F) \times \GL_1(F)$, we may write
\[  \theta(\pi) = \Pi \boxtimes \mu. \]
Note that the central character of $\pi$ is necessarily equal to $\mu$. Then we set
 \[ L(\pi) =  \phi_{\Pi} \times \mu: WD_F \longrightarrow \GL_4(\C) \times \GL_1(\C), \]
 where $\phi_{\Pi}$ is the Langlands parameter of $\Pi$.
We need to show that $L(\pi)$  factors through the inclusion $\iota: \GSp_4(\C) \hookrightarrow \GL_4(\C) \times \GL_1(\C)$.
\vskip 5pt

If $\pi$ is a discrete series representation as in Thm. \ref{T:summary}(iii)(a), then
$\Pi$ is a discrete series representation of $\GL_4(F)$ such that  $L(s,  \wedge^2\phi_{\Pi}  \otimes \mu^{-1})$ has a pole at $s = 0$. By Lemma \ref{L:par}, we conclude that $\phi_{\Pi} \times \mu$ factors through $\GSp_4(\C)$ and is a discrete series parameter.
On the other hand, if $\pi$ is a non-discrete series representation as in Thm. \ref{T:summary}(iii)(b), then we have
\[  \phi_{\Pi} = \phi \oplus \phi \cdot \chi \quad \text{and} \quad \mu = \det \phi \cdot \chi,\quad \text{with $\chi \ne 1$} . \]
One may consider the nondegenerate symplectic form which is totally isotropic on $\phi$ and $\phi \cdot \chi$ and such that the pairing between $\phi$ and $\phi \cdot \chi$ is given by the natural map
\[  \phi \otimes \phi \cdot \chi \longrightarrow \bigwedge^2 \phi \cdot \chi = \mu. \]
It is clear that this last map is $WD_F$-equivariant, so that $L(\pi)$ is a non-discrete series $L$-parameter of $\GSp_4$.
Thus,  we see again that $L(\pi)$ is a discrete series parameter iff $\pi$ is a discrete series representation.

\vskip 10pt

Using the fact that the theta correspondence preserves central characters and the basic properties of the local Langlands correspondence for $\GL_n$, one easily checks that the central character of $\pi$ is equal to the similitude character of $L(\pi)$, and for any character $\chi$, one has
\[  L(\pi \otimes \chi)  = L(\pi) \otimes \chi. \]
We have thus completed the definition of the map $L$ and verified properties (i), (iii) and (iv) of the Main Theorem.

\vskip 10pt
Observe that, in Thm. \ref{T:dichotomy}, there is another
partition of the set of irreducible representations of $\GSp_4(F)$ into two sets , namely those of Type (A) or (B). One could have defined the map $L$ using this partition as follows.
\vskip 5pt

\noindent{\bf Type A:} If $\pi$ is of Type (A), then $\pi = \theta(\tau^D_1 \boxtimes \tau^D_2)$ where now $D$ is the quaternion division algebra. One then defines $L(\pi)$ as in the Type (I) case above.
\vskip 5pt

\noindent{\bf Type B:}  If $\pi$ is of Type (B), then the theta lift of $\pi$ to $\GSO(V)$ is nonzero and has the form $\Pi \boxtimes \mu$. One then defines $L(\pi)$ following the Type (II) case above.

\vskip 5pt

The only potential difference in these two definitions is for those representations $\pi$ which have nonzero theta lifts to $\GSO(V_2) = \GSO_{2,2}(F)  $.  However,
it follows from Thm. \ref{T:summary}(iv) that the two definitions are in fact the same.

 \vskip 10pt

\noindent{\bf \underline{Surjectivity and Fibers}}
\vskip 5pt

For a given $L$-parameter
\[  \phi: WD_F \longrightarrow \GSp_4(\C), \]
with $\simi(\phi) = \mu$, we must now determine the fiber of the map $L$ over $\phi$, and in particular show that it is non-empty. From the construction of $L$, observe that the parameter $L(\pi)$ is endoscopic if and only if $\pi$ is of Type I. Thus, we see that $\# L_{\phi} \leq  2$ if $\phi$ is endoscopic, and $\# L_{\phi} \leq 1$ otherwise. We want to show that $\# L_{\phi} = \# A_{\phi}$, and we consider the endoscopic and non-endoscopic cases separately.
 \vskip 5pt

 \noindent {\bf Endoscopic case: } If $\phi = \phi_1 \oplus \phi_2$ is endoscopic, then $\phi$ gives rise to an $L$-parameter of $\GSO(V_2)$ and thus determines a representation $\tau_1 \boxtimes \tau_2$ of $\GSO(V_2)$, where $\tau_i$ has $L$-parameter $\phi_i$. If the $\phi_i$'s are irreducible, then $\phi$ is also an $L$-parameter of $\GSO(D)$ and thus determines a representation $\tau_1^D \boxtimes \tau_2^D$ of $\GSO(D)$, with $\tau_i^D$ the Jacquet-Langlands lift of $\tau_i$. By Thm. \ref{T:summary}(i) and (ii), both $\tau_1 \boxtimes \tau_2$ and $\tau_1^D \boxtimes \tau_2^D$ have nonzero theta lifts to $\GSp_4$, and it follows from the construction of $L$ that
 \[  L_{\phi} = \{ \theta(\tau_1 \boxtimes\tau_2), \theta(\tau_1^D \boxtimes \tau_2^D) \}, \]
where the latter representation is regarded as zero if one of the $\phi_i$'s is reducible.
In view of Lemma \ref{L:endoscopic}, we see that $\# L_{\phi} = \# A_{\phi}$.
Moreover, when $\# L_{\phi} =2$, we have $A_{\phi} = \Z/2\Z$,
which has two irreducible characters ${\bf 1}$ and ${\bf sign}$. In that case,  we set
\[
\pi_{\bf 1} = \theta(\tau_1 \boxtimes \tau_2) \quad \text{and} \quad  \pi_{\bf sign} =  \theta(\tau_1^D \boxtimes \tau_2^D).  \]
Note that the representation $\pi_{\bf 1}$ is generic, whereas $\pi_{\bf sign}$ is non-generic.

\vskip 5pt

\noindent {\bf Non-Endoscopic case:}   If $\phi$ is non-endoscopic with similitude character $\mu$, then $\phi$ is as described in Lemma \ref{L:endoscopic} or Thm. \ref{T:summary}(iii). In this case, Thm. \ref{T:summary}(iii) implies immediately that $L_{\phi}$ is non-empty, so that $\# L_{\phi} = 1$. Indeed, if $\Pi$ is the representation of $\GL_4(F)$ with $L$-parameter $\phi$, then the representation $\Pi \boxtimes \mu$ of $\GSO(V)$ has nonzero theta lift to $\GSp_4$ by Thm. \ref{T:summary}(iii), so that
\[  L_{\phi} = \{ \theta(\Pi \boxtimes \mu) \}. \]

\vskip 10pt

Thus, for each $L$-parameter $\phi$, we have determined the fiber $L_{\phi}$, which consists of one or two representations parametrized by the irreducible characters of $A_{\phi}$.
Indeed, in the companion paper \cite[Prop. 13.1]{GT4}, we explicitly write down the $L$-parameters of all non-supercuspidal representations.
We should mention that  for a non-discrete series parameter $\phi$,  a construction of the $L$-packet has been carried out by B. Roberts and R. Schmidt in their monograph \cite{RS} using the results of Sally-Tadic \cite{ST}.  The reader can refer to the table in \cite[\S A.5]{RS} for the explicit description. Using \cite[Prop. 13.1]{GT4}, it is easy to check that our definition of $L_{\phi}$ agrees with theirs.
\vskip 5pt

We have thus verified property (ii) in the Main Theorem.

\vskip 10pt

\noindent{\bf  \underline{Genericity}}
\vskip 5pt

Let us conclude this section by verifying property (vii) of the Main Theorem, which relates generic representations and the adjoint $L$-factor. For those $L$-packets which contain a non-supercuspidal representation, this is a straightforward calculation using the companion paper \cite{GT4} and the structure of principal series representations of $\GSp_4(F)$ \cite{ST}.  The verification in this case has been carried out in a recent paper of Asgari-Schmidt \cite{AS}. However, we give an independent proof in \cite[Prop. 13.2]{GT4}, as our  verification is much more concise than that of \cite{AS}.

\vskip 5pt
Suppose on the other hand that $L_{\phi}$ contains only supercuspidal representations. From our construction of the discrete series $L$-packets, it is clear that such an $L$-packet always contains a generic supercuspidal representation, since it contains an element which is the theta lift of a (necessarily generic) supercuspidal representation of $\GSO_{2,2}(F)  $ or $\GSO(V)=\GSO_{3,3}(F)$. Thus, we need to show that
$L(s, Ad \circ \phi)$ is holomorphic at $s=1$.  But we know that $\phi$ is either an irreducible 4-dimensional representation of $W_F$, or else $\phi = \phi_1 \oplus \phi_2$ where $\phi_1 \ne \phi_2$ are irreducible 2-dimensional representations of  $W_F$ with $\det \phi_1 = \det \phi_2$.  Thus, up to a twist by an unramified character, $\phi$ is of Galois type and has bounded image in $\GL_4(\C)$. Thus, $Ad \circ \phi$ is a representation of $W_F$ with bounded image. It follows that $L(s, Ad \circ \phi)$ is holomorphic in $\operatorname{Re}(s) > 0$.

\vskip 15pt

\section{\bf Preservation of Local Factors}  \label{S:factors}

In this section, we shall check that the local Langlands correspondence we defined in the previous section respects the $L$-, $\gamma$- and $\epsilon$-factors of pairs for generic representations and non-generic non-supercuspidal representations. This will prove property (v) of the Main Theorem.  More precisely, we have:
\vskip 10pt

\begin{Thm}  \label{T:localfac}
Let $\pi$ be an irreducible generic or non-supercuspidal representation of $\GSp_4(F)$ and $\sigma$ an irreducible representation of
$\GL_r(F)$.  Suppose that $\phi_{\pi}$ and $\phi_{\sigma}$ are the $L$-parameters of $\pi$ and $\sigma$ respectively. Then
\[ \begin{cases}
 L(s, \pi \times \sigma) = L(s, \phi_{\pi} \otimes \phi_{\sigma}), \\
\epsilon(s, \pi \times \sigma, \psi) = \epsilon(s, \phi_{\pi} \otimes \phi_{\sigma}, \psi), \\
\gamma(s, \pi \times \sigma, \psi) = \gamma(s, \phi_{\pi} \otimes \phi_{\sigma}, \psi).
\end{cases} \]
\end{Thm}

\begin{proof}
The proof is entirely similar to that of \cite[Prop. 5.4]{MS}.
By using the multiplicativity of $\gamma$-factors and the explicit determination of local theta correspondence given in the companion paper \cite{GT4}, we are reduced by a standard argument to the case when $\pi$ and $\sigma$ are generic supercuspidal representations.
\vskip 10pt

We first consider the case when the theta lift of $\pi$ to $\GSO_{2,2}(F)  $ is zero, so that its theta lift to $\GSO(V)$ is nonzero supercuspidal.
 By \cite[Lemma 5.2]{MS}, one can find a totally imaginary number field $\mathbb{F}$ such that
$\mathbb{F}_{v_0} = F$ for some place $v_0$. We may consider the split group $\GSp_4$ over $\mathbb{F}$.
By \cite[Prop. 5.1]{Sh}, one can find a globally generic cuspidal representation $\Pi$ of $\GSp_4(\A_{\mathbb{F}})$ such that  $\Pi_{v_0} \cong \pi$ and for all other finite places $v \ne v_0$, $\Pi_v$ is unramified. By \cite{GRS},
the global theta lift $\theta(\Pi)$ of $\Pi$ to $\GSO(V)$ is nonzero and cuspidal (since its local component $\Pi_{v_0}$ does not participate in the local theta correspondence with $\GSO_{2,2}(F)$).
Similarly, let $\Sigma$ be a cuspidal representation of $\GL_r(\A_{\mathbb{F}})$ such that
$\Sigma_{v_0} = \sigma$ and $\Sigma_v$ is unramified for all finite places $v \ne v_0$.
\vskip 10pt

By the functoriality of the theta correspondence for unramified representations at finite places (see Prop. \ref{P:unram}), we see that for all finite $v \ne v_0$,
\[ \begin{cases}
 L(s, \Pi_v \times \Sigma_v) = L(s, \theta(\Pi_v) \times \Sigma_v) = L(s, \phi_{\Pi_v} \otimes \phi_{\Sigma_v}) \\
\epsilon(s, \Pi_v \times \Sigma_v, \psi_v) = \epsilon(s, \theta(\Pi_v) \times \Sigma_v) = \epsilon(s, \phi_{\Pi_v} \otimes \phi_{\Sigma_v}, \psi_v) \\
\gamma(s, \Pi_v \times \Sigma_v, \psi_v) = \gamma(s, \theta(\Pi_v) \times \Sigma_v) = \gamma(s, \phi_{\Pi_v} \otimes \phi_{\Sigma_v}, \psi_v).
\end{cases} \]
Here, if we regard $\theta(\Pi_v)$ as a representation $\Pi'_v \boxtimes \mu_v$ of $\GL_4(\mathbb{F}_v) \times \GL_1(\mathbb{F}_v)$, then the $L$-factor $L(s, \theta(\Pi_v) \times \Sigma_v)$  is simply the Rankin-Selberg $L$-factor
$L(s, \Pi'_v \times \Sigma_v)$ (cf. (e) of Table \ref{table}).
Moreover, the theta correspondences over $\C$ have been completely determined by Adams-Barbasch \cite{AB}. Though they work with isometry groups, there is no subtlety in passing from isometry groups to similitude groups over $\C$. Thus, from their results, one sees that the theta correspondence from $\GSp_4$ to $\GSO(V)$ is functorial with respect to the inclusion
\[  \iota: \GSp_4(\C) \longrightarrow \GL_4(\C) \times \GL_1(\C). \]
Thus the above identities of local factors also hold over the archimedean places of $\mathbb{F}$.

\vskip 10pt
Now by \cite[Thm. 3.5 (3.14)]{Sh},  we see that for some finite set $S$ of places including all the archimedean ones, we have
\[  \begin{cases}
L^S(s, \Pi \times \Sigma) =  \left( \prod_{v \in S}  \gamma(s, \Pi_v \times \Sigma_v, \psi_v) \right) \cdot L^S(1-s, \Pi^{\vee} \times \Sigma^{\vee})  \\
L^S(s, \Theta(\Pi) \times \Sigma) =  \left( \prod_{v \in S} \gamma(s, \Theta(\Pi_v) \times \Sigma_v, \psi_v) \right) \cdot L^S(1-s, \Theta(\Pi^{\vee}) \times \Sigma^{\vee}).
\end{cases} \]
Taking everything into account, one deduces that at $v = v_0$
\[  \gamma(s, \Pi_{v_0} \times \Sigma_{v_0}, \psi_{v_0})=
 \gamma(s, \Theta(\Pi_{v_0}) \times \Sigma_{v_0}, \psi_{v_0}). \]
From the definition of the parameter $\phi_{\pi}$ and the fact that  the local Langlands correspondence for $\GL_4$ respects local factors of pairs, we conclude that
\[  \gamma(s, \pi \otimes \sigma, \psi) = \gamma(s, \phi_{\pi} \otimes \phi_{\sigma}, \psi). \]
Then as in \cite[Cor. 5.1]{MS}, since one knows that $L$-functions of discrete series representations are holomorphic in $\operatorname{Re}(s) > 0$,  we deduce the desired identity of $L$-factors from that of the $\gamma$-factors. From this, the identity of $\epsilon$-factors also follows. This proves the theorem for those supercuspidal $\pi$ which do not lift to $\GSO_{2,2}(F) $.
\vskip 10pt

Consider now the case when the theta lift of $\pi$ to $\GSO_{2,2}(F) $ is non-zero. Thus, there is a supercuspidal representation $\tau_1 \boxtimes \tau_2$ of $\GSO_{2,2}(F) $ such that $\theta(\tau_1 \boxtimes \tau_2) = \pi$. Now we can globalize $\tau_1 \boxtimes \tau_2$ as above and repeat the same argument to get
\[  \gamma(s, \pi \times \sigma, \psi) = \gamma(s, (\tau_1 \boxtimes \tau_2) \times \sigma, \psi) =  \gamma(s, \phi_{\pi} \otimes \phi_{\sigma},\psi). \]
 From this, the desired identities of $L$-factors and $\epsilon$-factors follow. The theorem is proven.
 \end{proof}
\vskip 10pt

\begin{Prop}
Suppose that one has a theory of $\gamma$-factors for all irreducible representations of
$\GSp_4(F) \times \GL_r(F)$ satisfying the expected properties, for example those listed in \cite[Thm. 4]{LR}. Then the conclusion of Theorem \ref{T:localfac} holds for all irreducible representations.
\end{Prop}

\begin{proof}
We shall only give a sketch proof.
As before, one is reduced to the case when $\pi$ is a non-generic supercuspidal representation and
$\sigma$ is a supercuspidal representation of $\GL_r$.  In this case, $\pi = \theta(\tau^D_1 \boxtimes \tau^D_2)$ for a representation $\tau^D_1 \boxtimes \tau^D_2$ of $\GSO(D)$ (where $D$ is the quaternion division algebra over $F$).
\vskip 5pt

Now as in the proof of Theorem \ref{T:localfac}, let $\mathbb{F}$ be a totally imaginary number field with
$\mathbb{F}_{v_0} = F$.  Let $v_1$ be another finite place of $\mathbb{F}$ and let
$\mathbb{D}$ be the quaternion algebra over $\mathbb{F}$ ramified at precisely
$v_0$ and $v_1$.  Pick an irreducible representation $\tau$ of $\mathbb{D}^{\times}(\mathbb{F}_{v_1})$ of dimension $> 1$, so that $\tau \boxtimes \tau$ is a representation of $\GSO(\mathbb{D})(\mathbb{F}_{v_1})$. One can then find a cuspidal representation $\Pi^{\mathbb{D}}$ of $\GSO(\mathbb{D})(\A_{\mathbb{F}})$ whose local components at $v_0$ and $v_1$ are $\tau^D_1 \boxtimes \tau^D_2$ and $\tau \boxtimes \tau$ respectively. Moreover, by the Jacquet-Langlands correspondence, one knows that $JL(\Pi^{\mathbb{D}})$ is cuspidal and hence the local components of $\Pi^{\mathbb{D}}$ are generic at all places outside of $\{ v_0, v_1\}$.
\vskip 5pt

One can show that the global theta lift $\Theta(\Pi^{\mathbb{D}})$ of $\Pi^{\mathbb{D}}$ to $\GSp_4(\A_{\mathbb{F}})$ is nonzero, irreducible and cuspidal (cf. \cite{GT3}).  Moreover, the local component of $\Theta(\Pi^{\mathbb{D}})$ is $\pi$ at $v_0$, non-supercuspidal at $v_1$ and generic at all other places. Thus,
one knows the desired equalities of local factors at every place except $v_0$,
and the same argument as in the proof of the Theorem \ref{T:localfac} gives the desired result at $v_0$.
\end{proof}
\vskip10pt

One can also consider the standard $L$-function of $\GSp_4 \times \GL_r$.
By the same argument as above, one has the following theorem; we omit the details.

\begin{Thm}
Let $\pi$ be a generic representation or a non-supercuspidal representation of $\GSp_4(F)$ and $\sigma$ a representation of
$\GL_r(F)$.  Suppose that $\phi_{\pi}$ and $\phi_{\sigma}$ are the $L$-parameters of $\pi$ and $\sigma$ respectively. Then
\[ \begin{cases}
 L(s, \pi \times \sigma, std \boxtimes std) = L(s, std \, \phi_{\pi} \otimes \phi_{\sigma}), \\
\epsilon(s, \pi \times \sigma, std \boxtimes std , \psi) = \epsilon(s, std \, \phi_{\pi} \otimes \phi_{\sigma}, \psi), \\
\gamma(s, \pi \times \sigma, std \boxtimes std , \psi) = \gamma(s, std \, \phi_{\pi} \otimes \phi_{\sigma}, \psi).
\end{cases} \]
If one has a theory of these standard $\gamma$-factors for all irreducible representations of $\GSp_4(F) \times \GL_4(F)$, then the above identities hold for all irreducible representations.
\end{Thm}

\vskip 15pt

\section{\bf Conservation of Plancherel Measure} \label{S:Plan}

In this section, we prove (vi) of the Main Theorem, which expresses the Plancherel measure $\mu(s, \pi \times \sigma, \psi)$ in terms of the product of various gamma factors of the corresponding $L$-parameters.
Let us briefly recall the definition of the relevant Plancherel measure.
\vskip 5pt

Let $\pi$ be an irreducible representation of $\GSp_4(F)$ and $\sigma$ a representation of $\GL_r(F)$, so that $\pi \boxtimes \sigma$ is a representation of
\[  M_r(F): =\GSp_4(F) \times \GL_r(F) \cong \GSpin_5(F) \times \GL_r(F). \]
Now $M_r$ is the Levi factor of a maximal parabolic subgroup $P_r = M_r \cdot N_r$ of
$G_r = \GSpin_{2r+5}$.  One can thus form the generalized principal series representation
\[  I_{P_r}(s, \pi \boxtimes \sigma) = \Ind_{P_r}^{G_r} \delta_{P_r}^{1/2} \cdot  \pi \boxtimes
\sigma |\det|^s, \]
where $\det$ is the determinant character of $\GL_r(F)$.
If $\bar{P}_r = M_r \cdot \bar{N}_r$ is the opposite parabolic, then we similarly have the induced representation $I_{\bar{P}_r}(s, \pi \boxtimes \sigma)$. The additive character $\psi$ determines a Haar measure on $N_r$, which induces a dual measure on $\bar{N}_r$. Then there is a standard intertwining operator
\[  A_{\psi}(s, \pi \boxtimes \sigma, N_r, \bar{N}_r) : I_{P_r}(s, \pi \boxtimes \sigma) \longrightarrow I_{\bar{P}_r}(s, \pi \boxtimes \sigma). \]
Then the composite
$A_{\psi}(s, \pi \boxtimes \sigma, \bar{N}_r, N_r) \circ A_{\psi}(s, \pi \boxtimes \sigma, N_r , \bar{N}_r)$
 is a scalar operator on $I_{P_r}(s, \pi \boxtimes \sigma)$ and the Plancherel measure is the scalar-valued meromorphic function defined by
\[  \mu(s, \pi \boxtimes \sigma,\psi)^{-1} = A_{\psi}(s, \pi \boxtimes \sigma, \bar{N}_r, N_r) \circ A_{\psi}(s, \pi \boxtimes \sigma, N_r , \bar{N}_r). \]
When $\pi$ and $\sigma$ are both supercuspidal, the analytic properties of $\mu(s, \pi \times \sigma,\psi)$ determine the points of reducibility of the principal series $I_{P_r}(s, \pi \boxtimes \sigma)$. More precisely, we have [Si, \S 5.3-5.4]:
\vskip 5pt

\begin{Prop}
Suppose that $\pi \boxtimes \sigma$ is a unitary supercuspidal representation of
$\GSp_4(F) \times \GL_r(F)$.
\vskip 5pt
(i) On the imaginary axis $i \mathbb{R}$, $\mu(s, \pi \boxtimes \sigma,\psi)$ is holomorphic and $\geq 0$.
\vskip 5pt

(ii) If $\sigma^{\vee} \ne \sigma \otimes \omega_{\pi}$, then $\mu(0, \pi \boxtimes \sigma,\psi)  \ne 0$ and $I_{P_r}(s,\pi \boxtimes \sigma)$ is irreducible for all $s \in \mathbb{R}$.
\vskip 5pt

(iii) If $\sigma^{\vee} = \sigma \otimes \omega_{\pi}$, then there is a unique real $s_0\geq 0$ such that $I_{P_r}(s_0, \pi \boxtimes \sigma)$ is reducible. Moreover, $s_0 > 0$ if and only if $\mu(0, \pi \boxtimes \sigma,\psi) =  0$, in which case $s_0$ is the unique pole of $\mu(s, \pi \boxtimes \sigma,\psi)$ on the positive real axis.
\end{Prop}
\vskip 10pt

When $\pi \boxtimes \sigma$ is a generic (not necessarily supercuspidal) representation,  Shahidi showed that the meromorphic function $\mu(s, \pi \boxtimes\sigma,\psi)$ can be expressed as a product of gamma factors.  As a result, he was able to determine with great precision the unique point of reducibility in (iii) of the above proposition when $\mu(0, \pi \boxtimes \sigma,\psi)=0$.
Let us recall his results for the case at hand.
\vskip 5pt

The parabolic subgroup $P_r \subset G_r$  gives rise to a dual parabolic subgroup $P_r^{\vee} = M_r^{\vee} \cdot N_r^{\vee}$ in the dual group $G_r^{\vee} = \GSp_{2r+4}(\C)$ and we have:
\[  M_r^{\vee} = \GSp_4(\C) \times \GL_{r}(\C). \]
Under the adjoint action of $M_r^{\vee}$, the Lie algebra $\mathfrak{n}_r^{\vee}$ of the unipotent radical $N_r^{\vee}$ decomposes as $\mathfrak{n}_r^{\vee} = r_1 \oplus r_2$ with
\[
r_1 = std^{\vee} \boxtimes std \quad \text{and} \quad
r_2 = \simi^{-1} \otimes  Sym^2 \]
where $std$ denotes the relevant standard representation and $\simi$ is the similitude character of
$\GSp_4(\C)$. On the opposite nilpotent radical $\bar{\mathfrak{n}}_r^{\vee}$, the adjoint action of $M_r^{\vee}$ is the dual representation $r_1^{\vee} \oplus r_2^{\vee}$.
Now  we have:
\vskip 5pt

\begin{Prop}
Suppose that $\pi \boxtimes \sigma$ is a generic representation of $M_r(F) = \GSp_4(F) \times \GL_r(F)$.
Then
\[  \mu(s, \pi \boxtimes \sigma,\psi) = \gamma(s, \pi \boxtimes \sigma, r_1,\psi) \cdot \gamma(s, \pi \boxtimes \sigma, r_1^{\vee}, \overline{\psi}) \cdot \gamma(2s, \pi \boxtimes \sigma, r_2 ,\psi) \cdot
 \gamma(2s, \pi \boxtimes \sigma, r_2^{\vee} ,\overline{\psi}), \]
which is in turn equal to
\[  \gamma(s, \phi_{\pi}^{\vee} \otimes \phi_{\sigma}, \psi) \cdot \gamma(-s, \phi_{\pi} \otimes \phi_{\sigma}^{\vee},\overline{\psi}) \cdot \gamma(2s, Sym^2 \phi_{\sigma} \otimes \simi \phi_{\pi}^{-1} ,\psi) \cdot
 \gamma(-2s, Sym^2 \phi_{\sigma}^{\vee} \otimes \simi \phi_{\pi} ,\overline{\psi})
 \]
\end{Prop}

\begin{proof}
The first equality is a result of Shahidi \cite[Thm. 3.5]{Sh}. The second equality follows
from Thm. \ref{T:localfac} and  a result of Henniart \cite[Thm. 1.4]{He2}.
Indeed, Henniart showed that the local Langlands correspondence for $GL_r$ respect the twisted symmetric square epsilon-factors up to multiplication by a root of unity $\alpha$. Hence
\[
 \gamma(2s, \pi \boxtimes \sigma, r_2 ,\psi) \cdot
 \gamma(-2s, \pi \boxtimes \sigma, r_2^{\vee} ,\overline{\psi})
 = \frac{  \gamma(2s, \pi \boxtimes \sigma, r_2 ,\psi)}{ \gamma(1+2s, \pi \boxtimes \sigma, r_2 ,\psi)}
 =\frac{\alpha \cdot \gamma(2s, Sym^2 \phi_{\sigma} \otimes \simi \phi_{\pi}^{-1} ,\psi)}{\alpha \cdot \gamma(1+ 2s, Sym^2 \phi_{\sigma} \otimes \simi \phi_{\pi}^{-1} ,\psi)}   \]
 which is in turn equal to
\[    \gamma(2s, Sym^2 \phi_{\sigma} \otimes \simi \phi_{\pi}^{-1} ,\psi) \cdot
 \gamma(-2s, Sym^2 \phi_{\sigma}^{\vee} \otimes \simi \phi_{\pi} ,\overline{\psi}). \]
 This explains why the root of unity $\alpha$ disappears. We thank Henniart for explaining this point to us.
 \end{proof}
\vskip 10pt

After these preliminaries, the main result of this section is:
\vskip 5pt

\begin{Thm} \label{T:Plan}
Let $\{ \pi, \pi'\}$ be an $L$-packet of $\GSp_4(F)$ such that $\pi'$ is non-generic supercuspidal
and $\pi$ is a generic discrete series representation. Then for any supercuspidal
representation $\sigma$ of $\GL_r(F)$,
\[  \mu(s, \pi \boxtimes \sigma,\psi) = \mu(s, \pi' \boxtimes \sigma,\psi). \]
In particular, if $\pi$ is also supercuspidal, then $I_{P_r}(s, \pi \boxtimes \sigma)$ is reducible if and only if $I_{P_r}(s, \pi' \boxtimes\sigma)$ is reducible.
\end{Thm}
\vskip 10pt

\begin{Cor}
Property (vi) of the Main Theorem holds.
\end{Cor}
\vskip 10pt

The rest of the section is devoted to the proof of Thm. \ref{T:Plan}, which is similar to the proof of \cite[Prop. 2.1]{MS2}. We know by the results of \S\ref{S:packets} that the representations $\pi$ and $\pi'$ can be obtained as theta lifts from $\GSO_{2,2}(F) $ and $\GSO(D)$ respectively, where $D$ is the quaternion division algebra over $F$. Thus, we have
\[
\pi = \theta(\tau_1 \boxtimes \tau_2) \quad \text{and} \quad \pi' = \theta(\tau^D_1 \boxtimes \tau^D_2), \]
where $\tau_i^D$ is a representation of $D^{\times}$ with the Jacquet-Langlands lift $\tau_i$ on $\GL_2(F)$.
Moreover, we know that $\tau_1 \ne \tau_2$.
\vskip 5pt

Now choose a number field $\mathbb{F}$ such that for two places $v_1$ and $v_2$, one has
$\mathbb{F}_{v_1} =\mathbb{F}_{v_2}=  F$. Let $\mathbb{D}$ be the quaternion division algebra over $\mathbb{F}$ which is ramified precisely at $v_1$ and $v_2$.   One can find a cuspidal representation $\Xi$ of $\GSO_{2,2}(\A_{\mathbb{F}}) $ such that $\Xi_{v_i} = \tau_1 \boxtimes \tau_2$ for $i =1$ or $2$.
If $\Xi^{\mathbb{D}}$ denotes the Jacquet-Langlands lift of $\Xi$ to
$\GSO(\mathbb{D})(\A_{\mathbb{F}})$, then  we have $\Xi_{v_i}^{\mathbb{D}} = \tau_1^D \boxtimes \tau_2^D$ for $i = 1$ or $2$. We may now consider the global theta lifts $\Pi  = \Theta(\Xi)$ and $\Pi' = \Theta(\Xi^{\mathbb{D}})$ of $\Xi$ and $\Xi^{\mathbb{D}}$ to $\GSp_4(\A_{\mathbb{F}})$. We have:

\vskip 10pt

\begin{Lem}
The global theta lifts $\Pi= \Theta(\Xi)$ and $\Pi' = \Theta(\Xi^{\mathbb{D}})$ are nonzero irreducible cuspidal representations of $\GSp_4(\A_{\mathbb{F}})$. Moreover, for $i =1$ or $2$,
\[  \Pi_{v_i} = \pi \quad \text{and} \quad \Pi'_{v_i} = \pi', \]
and for all $v \ne v_1$ or $v_2$, $\Pi_v = \Pi'_v$.
\end{Lem}

\begin{proof}
Since $\Xi$ is globally generic and $\tau_1 \ne \tau_2$, the non-vanishing and cuspidality of $\Pi = \Theta(\Xi)$ follows from \cite{GRS}.  On the other hand, the non-vanishing and cuspidality of $\Pi' = \Theta(\Xi^{\mathbb{D}})$  follows from [GT3].
\end{proof}

\vskip 10pt

Hence, we have two irreducible cuspidal representations $\Pi$ and $\Pi'$ on $\GSp_4(\A_{\mathbb{F}})$ which are locally isomorphic at all $v \ne v_1$ or $v_2$, and whose local components at $v_i$ ($i=1$ or $2$) are $\pi$ and $\pi'$ respectively.
\vskip 10pt

Similarly, let $\Sigma$ be a cuspidal representation of $\GL_r(\A_{\mathbb{F}})$ such that $\Sigma_{v_i} = \sigma$ for $i = 1$ and $2$. Now consider the global induced representations on $G_r(\mathbb{A}_{\mathbb{F}})$:
\[  I_{P_r}(s, \Pi \boxtimes \Sigma) \quad \text{and} \quad I_{P_r}(s, \Pi' \boxtimes \Sigma). \]
Fix an additive character $\Psi$ of $\mathbb{F} \backslash \A_{\mathbb{F}}$ such that $\Psi_{v_1} =\Psi_{v_2} =  \psi$. Then 
there are global standard intertwining operators
\[ A_{\Psi}(s, \Pi \boxtimes \Sigma, N_r, \bar{N}_r) : I_{P_r}(s, \Pi  \boxtimes \Sigma) \longrightarrow I_{\bar{P}_r}(s, \Pi \boxtimes \Sigma) \]
and
\[  A_{\Psi}(s, \Pi' \boxtimes \Sigma, N_r, \bar{N}_r) : I_{P_r}(s,\Pi' \boxtimes \Sigma) \longrightarrow I_{\bar{P}_r}(s, \Pi'  \boxtimes \Sigma), \]
which satisfy the functional equations
\[ A_{\Psi}(s, \Pi \boxtimes \Sigma, \bar{N}_r, N_r)  \circ A_{\Psi}(s, \Pi \boxtimes \Sigma, N_r, \bar{N}_r)
 = \operatorname{Id}  \]
and
\[ A_{\Psi}(s, \Pi' \boxtimes \Sigma, \bar{N}_r, N_r)  \circ A_{\Psi}(s, \Pi' \boxtimes \Sigma, N_r, \bar{N}_r)  = \operatorname{Id}.  \]
These global intertwining operators are actually independent of the choice of $\Psi$, but their decompositions into local intertwining operators (for $\text{Re}(s)$ large) depend on $\Psi$.
By comparing the two global functional equations and using the fact that
$\Pi_v = \Pi'_v$ for all $v \ne v_1$ or $v_2$, we deduce that
\[  \mu(s, \pi \boxtimes \sigma,\psi)^2 = \mu(s,\pi' \boxtimes \sigma,\psi)^2 \quad \text{for all $s \in \C$.} \]
However, the two Plancherel measures are $\geq 0$ when $s$ is purely imaginary. So we have the desired equality:
\[  \mu(s, \pi \boxtimes \sigma,\psi) = \mu(s,\pi' \boxtimes \sigma,\psi). \]
This completes the proof of Thm. \ref{T:Plan}.

\vskip 10pt

In a similar fashion, the group $M_r = \GSp_4 \times \GL_r$ may be considered as a Levi subgroup of a maximal parabolic subgroup $P'_r $ of $G'_r = \GSp_{2r+4}$. Thus,
given a representation $\pi \boxtimes \sigma$ of $M_r(F)$, one may consider the generalized principal series representation $I_{P'_r}(s,\pi \boxtimes \sigma)$ and the associated Plancherel measure $\mu'(s, \pi \boxtimes \sigma,\psi)$. The same argument as above gives:
\vskip 5pt

\begin{Thm} \label{T:Plan2}
Let $\{ \pi, \pi'\}$ be an $L$-packet of $\GSp_4(F)$ such that $\pi'$ is non-generic supercuspidal
and $\pi$ is a generic discrete series representation. Then for any supercuspidal
representation $\sigma$ of $\GL_r(F)$,
\[  \mu'(s, \pi \boxtimes \sigma,\psi) = \mu'(s, \pi' \boxtimes \sigma,\psi) \]
and these are in turn given by:
 \[  \gamma(s, std \phi_{\pi}^{\vee} \otimes \phi_{\sigma}, \psi) \cdot \gamma(-s, std \phi_{\pi} \otimes \phi_{\sigma}^{\vee},\overline{\psi}) \cdot \gamma(2s, \bigwedge^2 \phi_{\sigma}  ,\psi) \cdot
 \gamma(-2s, \bigwedge^2 \phi_{\sigma}^{\vee} ,\overline{\psi}). \]
 \end{Thm}

\vskip 10pt

\section{\bf Characterization of the Map $L$}  \label{S:char}

In this section, we shall show that the map $L$ is
characterized by the properties (i), (iii), (v) and (vi) in the Main Theorem. In particular, this will complete the proof of the Main Theorem. More precisely, we show:
\vskip 5pt

\begin{Thm} \label{T:unique}
There is at most one map
\[ L : \Pi(\GSp_4) \longrightarrow \Phi(\GSp_4) \]
satisfying:
\vskip 5pt

\noindent (a) the central character $\omega_{\pi}$ of $\pi$ corresponds to the similitude character $\simi(\phi_{\pi})$ of $\phi_{\pi} := L(\pi)$ under local class field theory;
\vskip 5pt

\noindent (b) $\pi$ is an essentially discrete series representation if and only if $\phi_{\pi}$ does not factor through any proper Levi subgroup of $\GSp_4(\C)$;

\vskip 5pt

\noindent (c) if $\pi$ is generic or non-supercuspidal, then for any irreducible representation
$\sigma$ of $\GL_r(F)$ with $r \leq 2$,
\[  \begin{cases}
L(s, \pi \times \sigma) = L(s, \phi_{\pi} \otimes \phi_{\sigma}), \\
\epsilon(s, \pi \times \sigma, \psi) = \epsilon(s, \phi_{\pi} \otimes \phi_{\sigma}, \psi). \end{cases} \]
\vskip 5pt

\noindent (d) if $\pi$ is  non-generic supercuspidal, then for any supercuspidal representation $\sigma$ of $\GL_r(F)$ with $r \leq 2$, the Plancherel measure $\mu(s, \pi \boxtimes \sigma)$ is equal to
\[  \gamma(s, \phi_{\pi}^{\vee} \otimes \phi_{\sigma}, \psi) \cdot \gamma(-s, \phi_{\pi} \otimes \phi_{\sigma}^{\vee},\overline{\psi}) \cdot \gamma(2s, Sym^2 \phi_{\sigma} \otimes \simi\phi_{\pi}^{-1}, \psi)
\cdot \gamma(-2s, Sym^2 \phi_{\sigma}^{\vee} \otimes \simi \phi_{\pi} ,\overline{\psi}). \]
 \end{Thm}
\vskip 10pt

The rest of the section is devoted to the proof of this theorem.
Recall from Lemma \ref{L:par} that the natural inclusion
\[  \iota: \GSp_4(\C) \longrightarrow \GL_4(\C) \times \GL_1(\C)  \]
gives rise to an injection
\[  \iota_*: \Phi(\GSp_4) \hookrightarrow \Phi(\GL_4) \times\Phi(\GL_1). \]
where the projection to $\Phi(\GL_1)$ is simply given by taking similitude character.
Hence, if $L_1$ and $L_2$ are two maps satisfying the requirements of the theorem,
it suffices to show that $\iota_*  \circ L_1 = \iota_* \circ L_2$. In view of the requirement (a) in the theorem, it suffices to show that for each $\pi \in \Pi(\GSp_4)$ with $\phi_i := L_i(\pi)$ ($i=1$ or $2$), $\phi_1$ and $\phi_2$ are equivalent as 4-dimensional representations of the Weil-Deligne group $WD_F$. We treat two different cases separately.
 \vskip 15pt

\noindent{\bf \underline{Case 1: $\pi$ is generic or non-supercuspidal}}
\vskip 5pt

In this case,  the requirement (c) implies that
\begin{equation} \label{E:Lfactor}
  \begin{cases}
L(s, \phi_1 \otimes \phi_{\sigma})= L(s, \phi_2 \otimes \phi_{\sigma}), \\
\epsilon(s, \phi_1 \otimes \phi_{\sigma}, \psi)= \epsilon(s, \phi_2 \otimes \phi_{\sigma}, \psi) \end{cases} \end{equation}
for any representation $\sigma$ of $\GL_r(F)$ with $r \leq 2$.
At this point, we remark that if the above equalities are assumed to hold for $r \leq 3$, then the results of Henniart in \cite[Cor. 1.4 and Thm. 1.7]{He3} would immediately imply that $\phi_1 \cong \phi_2$, as desired. However, since we are only requiring the above equalities to hold for $r \leq 2$, we need a refinement of these results of Henniart.

\vskip 10pt

The first refinement (which refines \cite[Cor. 1.4 and Cor. 1.9]{He3}) is the so-called $n \times (n-2)$ local converse theorem for $\GL_n$ which was shown in the thesis of J.-P. Chen and has recently appeared in print \cite{C}. We are of course interested in the case $n=4$.
Applying \cite[Thm. 1.1]{C}, we see that if $\phi_1$ and $\phi_2$ are both irreducible representations of the Weil-Deligne group $WD_F$, then equation (\ref{E:Lfactor}) implies that $\phi_1\cong \phi_2$.
\vskip 5pt

We thus need to consider the case when, without loss of generality, $\phi_1$ is reducible. Our argument below is nothing more than an elaboration of the proof of \cite[Thm. 1.7]{He3}, taking into account the key observation that since $\phi_1$ and $\phi_2$ factor through $\GSp_4(\C)$, they do not contain any irreducible 3-dimensional constituent.
\vskip 5pt

To be more precise, any irreducible representation of $WD_F$ is of the form $\rho \boxtimes S_r$ where $\rho$ is an irreducible representation of the Weil group $W_F$ and $S_r$ is the irreducible $r$-dimensional representation of $\SL_2(\C)$. If $\phi$ is a semisimple representation of $WD_F$, then for each $\rho \boxtimes S_r$, let $m_{\phi}(\rho,r)\geq 0$ be the multiplicity of $\rho \boxtimes S_r$ in $\phi$. Moreover, let $\nu_{\phi}(\rho,r) \geq 0$ denote the order of poles at $s = 0$ of the local $L$-factor $L(s, \phi \otimes (\rho \boxtimes S_r)^{\vee})$. We make a few simple observations:
\vskip 5pt

\begin{itemize}
\item[(i)] From the definitions, one has:
\[ L(s, \rho \boxtimes S_r) = L(s+ \frac{r-1}{2}, \rho). \]
\vskip 5pt

\item[(ii)] With $t = \min(r,k)$,
\[  L(s, (\rho \boxtimes S_r) \otimes (\varphi \boxtimes S_k)^{\vee})
= \prod_{j=0}^{t-1} L(s + \frac{r+k-2-2j}{2}, \rho \otimes \varphi^{\vee}). \]
\vskip 5pt

\item[(iii)] For any semisimple representation $\phi$ of $WD_F$,
\[  \nu_{\phi}(\rho, r) = \sum_{k\geq 1} \sum_{j = 1}^{\min(r,k)}  m_{\phi}(\rho |-|^{-(\frac{r+k}{2} - j)} , k). \]
\vskip 5pt

\item[(iv)] If the $\phi_1$ and $\phi_2$ above have a common constituent, then $\phi_1 \cong \phi_2$.
Indeed, after cancelling this common constituent, we have two representations $\phi_1'$ and $\phi_2'$ of dimension $<4$. Moreover,  $L$- and $\epsilon$-factors behave multiplicatively with respect to direct sums of representations. Hence, the analog of equation (\ref{E:Lfactor}) holds for $\phi'_i$ with $r \leq 2$. Thus, one can apply \cite[Cor. 1.4 and Thm. 1.7]{He3} to conclude that $\phi'_1 \cong \phi'_2$.
\end{itemize}
 \vskip 10pt

For $i = 1$ or $2$,  let us write $m_i(\rho, r)$ and  $\nu_i(\rho,r)$ for  $m_{\phi_i}(\rho, r)$ and $\nu_{\phi_i}(\rho,r)$ respectively.
 Now using the above observations (see also [He3, \S 4.2]), one sees that for any 1-dimensional character $\chi$ of $W_F$,  one has the following system of equations:
\[  \begin{cases}
\nu_i(\chi,2) = m_i(\chi|-|^{-1/2}, 1) + m_i(\chi ,2) + m_i(\chi |-|^{-1},2) +
m_i(\chi|-|^{-1},4) + m_i(\chi|-|^{-2},4), \\
\nu_i(\chi|-|^{-1/2}, 1) = m_i(\chi|-|^{-1/2},1) + m_i(\chi|-|^{-1},2) + m_i(\chi|-|^{-2}, 4), \\
\nu_i(\chi|-|^{1/2},1) = m_i(\chi|-|^{1/2},1) + m_i(\chi,2) + m_i(\chi|-|^{-1},4). \end{cases} \]
In particular, this implies that
\[  m_i(\chi,1) = \nu_i(\chi,1) + \nu_i(\chi|-|^{-1},1) - \nu_i(\chi|-|^{-1/2},2). \]
Thus, using the requirement (c), we conclude that any 1-dimensional irreducible constituent of $\phi_1$ occurs with the same multiplicity in $\phi_2$. In particular, if $\phi_1$ (and thus $\phi_2$) contains a 1-dimensional irreducible constituent,  we are done by observation (iv) above.
\vskip 10pt

We are thus reduced to the case where $\phi_1$ is the sum of two (possibly equivalent) irreducible 2-dimensional representations. If $\rho$ is an irreducible 2-dimensional representation of $W_F$, then
\[  \nu_i(\rho, 1) = m_i(\rho, 1) + m_i(\rho|-|^{-1/2}, 2). \]
 Thus, if $\phi_1$ contains $\rho$ and $\phi_2$ does not, we must have $\phi_2 = \rho|-|^{-1/2}
 \boxtimes S_2$ and $\nu_i(\rho,1) = m_1(\rho,1) = 1$. This would imply that $\phi_1 = \rho  \oplus \rho_0$ with $\rho_0 \ne \rho$, and thus $L(s, \phi_1 \otimes \rho_0^{\vee})$ has a pole at $s = 0$ whereas $L(s, \phi_2 \otimes \rho_0^{\vee})$ is holomorphic at $s=0$. With this contradiction, we see that if $\phi_1$ contains $\rho$, so must $\phi_2$.  Then by observation (iv) above, we have $\phi_1\cong \phi_2$.
 \vskip 10pt

 Finally, we are reduced to the case where
 \[  \phi_1 = (\chi \boxtimes S_2) \oplus (\mu \boxtimes S_2), \]
and $\phi_2$ is either of the same form as $\phi_1$ or is irreducible.
In this case, one has
\[
\nu_i(\chi|-|^{1/2},1) = m_i(\chi,2) + m_i(\chi|-|^{-1},4).  \]
Thus, if $\phi_2$ does not contain $\chi \boxtimes S_2$, then we must have $\phi_2 = \chi|-|^{-1} \boxtimes S_4$ and $m_1(\chi,2) =1$. This implies that $\mu \ne \chi$ and one has
\[   \nu_1(\mu,2) = 1 \quad \text{and} \quad \nu_2(\mu,2) = 0 \]
 which is a contradiction. Hence, $\phi_2$ must also contain $\chi \boxtimes S_2$ and by observation (iv) again, we deduce that $\phi_1 \cong \phi_2$.
  \vskip 15pt

\noindent{\bf \underline{Case 2: $\pi$ is non-generic supercuspidal}}
\vskip 5pt

In this case, we know that
$\pi = \theta(\tau_1^D \boxtimes \tau_2^D)$ for a representation $\tau_1^D \boxtimes \tau_2^D$ of $\GSO(D)$. If $\tau_i$ is the Jacquet-Langlands lift of $\tau_i^D$, then
set
\[  \Phi = \phi_{\tau_1} \oplus \phi_{\tau_2}. \]
We have shown in the previous section that $\mu(s, \pi \boxtimes \sigma)$ is equal to
\[  \gamma(s, \Phi \otimes \phi_{\sigma}, r_1,\psi) \cdot \gamma(s,\Phi \otimes \phi_{\sigma}, r_1^{\vee},\psi) \cdot \gamma(2s, \pi \boxtimes \sigma, r_2,\psi) \cdot \gamma(2s, \pi \boxtimes \sigma, r_2^{\vee},\psi). \]
Hence, if $L$ is a map satisfying the conditions of the theorem with $\phi_{\pi} := L(\pi)$, then
the requirement (d) implies that
\begin{equation} \label{E:gamma}
  \gamma(s, \phi_{\pi} \otimes \phi_{\sigma}, r_1,\psi) \cdot \gamma(s, \phi_{\pi} \otimes \phi_{\sigma}, r_1^{\vee},\psi) = \gamma(s, \Phi \otimes \phi_{\sigma}, r_1, \psi) \cdot \gamma(s,\Phi \otimes \phi_{\sigma}, r_1^{\vee}, \psi). \end{equation}
We need to show that this forces $\phi_{\pi}$ to be equal to $\Phi$.
\vskip 10pt

Suppose first that $\tau_1$ is supercuspidal.
Taking $\sigma = \tau_1$, the RHS of equation (\ref{E:gamma}) has a zero at $s = 0$, which implies that $L(-s, \phi_{\pi} \otimes \phi_{\tau_1}^{\vee}) \cdot L(s, \phi_{\pi}^{\vee} \otimes \phi_{\tau_1})$ has a pole at $s = 0$. Thus, for some irreducible constituent $\phi = \rho \boxtimes S_r$ of $\phi_{\pi}$,  the function
\[ L(-s, \phi \otimes \phi_{\tau_1}^{\vee}) \cdot L(s, \phi^{\vee}\otimes \phi_{\tau_1}) \]
has a pole at $s=0$. Since $\tau_1$ is supercuspidal, this occurs if and only if
\[  \rho = \phi_{\tau_1} \otimes |-|^{\pm (r-1)/2}. \]
Hence, we conclude that $\phi_{\pi}$ must contain  $\phi_{\tau_1}^{\vee}\cdot |-|^{\pm (r-1)/2} \boxtimes S_r$ as a constituent and further $r =1$ or $2$.  In other words, either $\phi_{\pi}$ contains $\phi_{\tau_1}$ or else $\phi_{\pi}$ is equal to the irreducible representation $\phi_{\tau_1} |-|^{\pm 1/2} \boxtimes S_2$.
\vskip 5pt

On the other hand, suppose that $\tau_1$ is the twisted Steinberg representation  $St_{\chi}$ so that
$\phi_{\tau_1} = \chi \boxtimes S_2$. Taking $\sigma = \chi$, we see that the RHS  of equation
(\ref{E:gamma}) has a zero at $s = 1/2$. Thus, for some constituent $\phi = \rho \boxtimes S_r$ of
$\phi_{\pi}$, the function
\[
\gamma(s, \rho^{\vee} \cdot \chi \boxtimes S_r) \cdot \gamma(-s, \rho \cdot \chi^{-1} \boxtimes S_r)
= (\text{$\epsilon$-factors}) \cdot \frac{L(\frac{r-1}{2} + 1-s, \rho \cdot \chi^{-1}) \cdot L(\frac{r-1}{2} + 1+ s, \rho^{\vee} \cdot \chi)}{L(\frac{r-1}{2}+s, \rho^{\vee} \cdot \chi) \cdot L(\frac{r-1}{2} -s, \rho \cdot \chi^{-1})} \]
has a zero at $s = 1/2$. Thus, the denominator must have a pole at $s = 1/2$.
This occurs if and only if
\[  \rho = \chi\cdot |-|^{r/2}  \quad \text{or} \quad \rho  = \chi \cdot |-|^{-(r-2)/2}. \]
Now by requirement (b) of the theorem, we know that $\phi_{\pi}$ is either irreducible or is the sum of two inequivalent 2-dimensional irreducible representations. Thus, $r = 2$ or $4$ above. In other words,
$\phi_{\pi}$ is either the irreducible 4-dimensional representation $S_4$ up to twisting, or else
$\phi_{\pi}$ contains $\chi|-| \boxtimes S_2$ or $\chi \boxtimes S_2$.

\vskip 5pt

Taking note of the fact  that $\tau_1 \ne \tau_2$ but $\omega_{\tau_1} = \omega_{\tau_2}$,
the above discussion shows that $\phi_{\pi}$ must be the sum of two irreducible 2-dimensional representations of $WD_F$, with each $\tau_i$ contributing a constituent
equal to $\phi_{\tau_i}$ or $\phi_{\tau_i} \cdot |-|$. Moreover, the latter can only occur if $\tau_i$ is non-supercuspidal. However, it is not difficult to check that equation (\ref{E:gamma}) cannot hold for all
$\sigma$ if $\phi_{\pi}$ contains $\phi_{\tau_i} \cdot |-|$ as a constituent.
The proof of Theorem \ref{T:unique} is complete.
\vskip 10pt

\vskip 5pt

\vskip 10pt

\noindent{\bf Remarks:} Observe that the requirement (b) is only used in the analysis of Case 2 above.
If the requirement (c) in the theorem is known to hold for all
representations $\pi$, then the proof given above in Case 1 shows that the map $L$ is characterized by (a) and (c) alone.

 \vskip 15pt
\section{\bf Comparison with Construction of Vigneras}  \label{S:vigneras}

When $p \ne 2$, Vigneras \cite{V} and Roberts \cite{Ro2} have given an alternative definition of $L$-packets. In this section, we verify that their definition agrees with ours.
For that, we shall begin by recalling their construction.
\vskip 10pt

We shall focus on an $L$-parameter $\phi$ for $\GSp_4$ which is trivial on $\SL_2(\C)$  and such that $\phi$ is irreducible of the form
\[   \phi = ind_{W_K}^{W_F} \phi_{\rho}. \]
Here $K/F$ is a quadratic field extension and $\phi_{\rho}$ is an irreducible 2-dimensional representation of $W_K$ which does not extend to $W_F$ but such that $\det \phi_{\rho}$ does extend.
In this case, there are two extensions of $\det \phi_{\rho}$ to $W_F$, differing from each other by twisting by $\omega_{K/F}$. One of these extensions of $\det \phi_{\rho}$ is precisely the similitude character $\mu$ of $\phi$. This is the crucial case treated in \cite{V}.
 \vskip 5pt

To construct the $L$-packet associated to $\phi$ following \cite{V}, consider the rank 4 quadratic spaces
\[
V^+_K = (K, \mathbb{N}_K) \oplus \mathbb{H} \quad \text{and} \quad
V_K^- = (K, \delta \cdot \mathbb{N}_K) \oplus \mathbb{H}, \]
where $\delta \in F^{\times} \smallsetminus \mathbb{N}_K(K^{\times})$.
The associated orthogonal similitude groups of these two quadratic spaces are isomorphic:
\[  \GSO(V_K^{\epsilon}) \cong
 (\GL_2(K) \times F^{\times})/ \Delta K^{\times},  \]
 where $\Delta(a) = (a, \mathbb{N}_K(a)^{-1})$ for $a \in K^{\times}$.
 \vskip 5pt

If $\rho$ is the supercuspidal representation of $\GL_2(K)$ associated to $\phi_{\rho}$, then
$\rho$ is not obtained via base change from $\GL_2(F)$.
One may consider the representation  $\rho  \boxtimes \mu$ of $\GSO(V_K^{\epsilon})$.
Considering the theta correspondence for $\GSO(V^{\epsilon}_K) \times \GSp_4(F)^+$, one gets two irreducible representations $\theta^{\epsilon}(\rho \boxtimes \mu)$ of $\GSp_4(F)^+$, which are exchanged under the conjugation action of an element of $\GSp_4(F) \smallsetminus \GSp_4(F)^+$.  Thus we have an irreducible representation
\[  V(\rho \boxtimes \mu) = ind_{\GSp_4^+}^{\GSp_4} \theta^{\epsilon}(\rho \boxtimes \mu). \]
The packet associated to $\phi$ by Vigneras consists of the single supercuspidal representation $V(\rho \boxtimes \mu)$.

\vskip 10pt

On the other hand, if $\Pi_{\rho}$ is the supercuspidal representation of $\GL_4(F)$ attached to
$\phi$, then $\Pi_{\rho}$ is the automorphic induction of $\rho$ from $\GL_2(K)$ to $\GL_4(F)$.
According to our definition of $L$-packets,
the packet associated to $\phi$ is the theta lift of $\Pi_{\rho} \boxtimes \mu$ to $\GSp_4(F)$. Thus, to show that our definition is consistent with that of Vigneras and Roberts, it suffices to show:
\vskip 5pt

\begin{Prop}
Under the theta lifting from $\GSp_4(F)$ to $\GSO(V)$, we have
\[  \theta(V(\rho \boxtimes \mu)) = \Pi_{\rho} \boxtimes \mu. \]
\end{Prop}
\vskip 5pt

\begin{proof}
We shall show the proposition using global means. Pick a quadratic extension $\mathbb{K}/\mathbb{F}$ of number fields such that for some place $v$ of $\mathbb{F}$, we have $\mathbb{K}_v/ \mathbb{F}_v = K/F$. Consider the quadratic space $V_{\mathbb{K}}  = (\mathbb{K}, \mathbb{N}_{\mathbb{K}}) \oplus \mathbb{H}$ over $\mathbb{F}$ so that
\[  \GSO(V_{\mathbb{K}}) \cong (\GL_2(\mathbb{K}) \times \mathbb{F}^{\times})/\Delta \mathbb{K}^{\times}. \]
Let $\Xi \boxtimes \Upsilon$ be a cuspidal representation of $\GSO(V_{\mathbb{K}})(\A_{\mathbb{K}})$ whose local component at the place $v$ is isomorphic to $\rho \boxtimes \mu$. Consider the global theta lift
$V(\Xi \boxtimes \Upsilon)$ of $\Xi \boxtimes \Upsilon$ to $\GSp_4(\A_{\mathbb{K}})$. Then $V(\Xi \boxtimes \Upsilon)$  is a nonzero globally generic cuspidal representation with central character $\Upsilon$ and the local component  at $v$ of an irreducible constituent of $V(\Xi \boxtimes \Upsilon)$  is the representation $V(\rho \boxtimes \mu)$.
\vskip 5pt

Now consider the theta lift of $V(\Xi \boxtimes \Upsilon)$ from $\GSp_4(\A_{\mathbb{F}})$ to $\GSO(V)(\A_{\mathbb{F}})$. It is not difficult to see that this theta lift is nonzero and cuspidal. Thus we obtain a cuspidal representation $\Sigma \boxtimes \Upsilon$ on $\GL_4(\A_{\mathbb{F}}) \times \GL_1(\A_{\mathbb{F}})$, whose local component at $v$ is the representation
$\theta(V(\rho \boxtimes \mu))$.
\vskip 5pt

From the functoriality of theta correspondences for unramified representations (Prop. \ref{P:unram}), we see that $\Sigma$ is nearly equivalent to the automorphic induction of $\Xi$ from $\GL_2(\A_{\mathbb{K}})$ to $\GL_4(\A_{\mathbb{F}})$.
By the strong multiplicity one theorem for $\GL_4$, these two representations are thus isomorphic. In particular, the proposition follows by extracting the local component at $v$.
\end{proof}
\vskip 5pt

\begin{Cor}
When $p \ne 2$, Vigneras' construction of $L$-packets in \cite{V} exhausts all irreducible representations of $\GSp_4(F)$.
\end{Cor}

\vskip 10pt

\section{\bf Global Generic Lifting from $\GSp_4$ to $\GL_4$}  \label{S:global}

We conclude this paper with a global consequence of our local results. Hence, in this section, $F$ will denote a number field and $\A$ its ring of adeles.  Jacquet, Piatetski-Shapiro and Shalika announced some thirty years ago that they  have used the global theta correspondence for $\GSp_4 \times \GSO(V)$ to show the weak lifting of globally generic cuspidal representations of
$\GSp_4(\A)$ to $\GL_4(\A)$, but  unfortunately their proof was never published. Recently, an alternative proof of this weak lifting, with some further refinements, was given by Asgari-Shahidi \cite{ASh} using the converse theorem.
\vskip 5pt

Though the results of Jacquet-Piatetski-Shapiro-Shalika were not published, essentially all the details of their proof can be found in the papers \cite{So} and \cite{GRS}.  Namely, given a globally generic cuspidal representation $\pi$ of $\GSp_4(\A)$, one considers its global theta lift $\Theta(\pi)$ to $\GSO(V)$ and shows that the Whittaker-Fourier
coefficients of $\Theta(\pi)$  can be expressed
in terms of the Whittaker coefficient of $\pi$ \cite[Prop. 2.2 and Cor. 2.5]{GRS}. Thus one concludes that $\Theta(\pi)$ is globally generic (and in particular is nonzero). Moreover, by the functoriality of the local theta correspondence for unramified representations (Prop. \ref{P:unram}), one sees that  $\Theta(\pi)$ is a weak lift of $\pi$.
\vskip 5pt

The purpose of this last section is to strengthen this weak lifting to a strong one:
\vskip 5pt

\begin{Thm}
The global theta lifting $\pi \mapsto \Theta(\pi)  = \Pi \boxtimes \mu$ defines an injection from the set of globally generic cuspidal representations of $\GSp_4(\A)$ to the set of globally generic automorphic representations $\Pi \boxtimes \mu$ of $\GL_4(\A) \times \GL_1(\A)$. Moreover, one has:
\vskip 5pt

\noindent (i) $\mu = \omega_{\pi}$ and the central character of $\Pi$ is $\omega_{\pi}^2$;
\vskip 5pt

\noindent (ii) $\Pi \cong \Pi^{\vee} \otimes  \omega_{\pi}$;

\vskip 5pt
\noindent (iii) The image of the lifting consists precisely of those automorphic representations $\Pi \boxtimes \mu$ satisfying one of the following:
\vskip 5pt

\begin{itemize}
\item[(a)] $\Pi$ is cuspidal and the (partial) twisted exterior square $L$-function $L^S(s, \Pi, \bigwedge^2 \otimes \mu^{-1})$ has a pole at $s = 1$;
\vskip 5pt

\item[(b)] $\Pi$ is  an isobaric sum $\tau_1 \boxplus \tau_2$, where $\tau_1 \ne \tau_2$  are cuspidal representations of $\GL_2(\A)$ such that $\tau_i = \tau_i^{\vee} \otimes \mu$. In this case, $\pi$ is the global theta lift of the cuspidal representation $\tau_1 \boxtimes \tau_2$ of $\GSO_{2,2}(\A) $.
\end{itemize}
\vskip 5pt

\noindent (iv) for each place $v$ of $F$, one has an equality of $L$-parameters
\[  \phi_{\pi,v} = \phi_{\Pi,v} : WD_{F_v} \longrightarrow \GL_4(\C). \]
\vskip 5pt
\noindent In other words, $\pi \mapsto \Theta(\pi)$ is a {\bf strong} lift and one has the following equalities of global $L$-functions and $\epsilon$-factors:
\[
 L(s, \pi \times \Sigma)   = L(s, \Pi \times \Sigma) \quad \text{and} \quad
\epsilon(s, \pi \times \Sigma) = \epsilon(s, \Pi \times \Sigma) \]
for any cuspidal representation $\Sigma$ of $\GL_r(\A)$.

\end{Thm}
 \vskip 5pt

We remark that (i), (ii) and parts of (iii) constitute the main theorem of \cite{ASh} and are proved by entirely different methods there. More precisely, [ASh] showed that the image of the global generic lifting is contained in the set of automorphic representations satisfying the conditions (a) and (b) of (iii). Thus the new results in our theorem are (iv) and the other half of (iii), namely that any automorphic representation satisfying (a) and (b) of (iii) is in the image.

\vskip 5pt

The rest of the section is devoted to the proof of the theorem. By the local results of this paper, the local theta lift for $\GSp_4 \times \GSO(V)$ satisfies Howe's conjecture. The analogous result at archimedean places has been proved by Adams-Barbasch \cite{AB} over $\C$ and by A. Paul [Pa] over $\R$ (suitably extended to the similitude case).
Thus the map $\pi \mapsto \Theta(\pi) = \Pi \boxtimes\mu$ is injective.
Since the local theta correspondence preserves central characters, we immediately deduce (i).
Moreover, for a finite place, (iv)
 follows directly from our definition of $L$-parameters for $\GSp_4$ and Thm. \ref{T:summary}(iii)(b). For an infinite place, (iv) follows by \cite{AB} and \cite{Pa}. Thus, the weak lifting is strong.
\vskip 5pt

By (iv) and the property of the local Langlands correspondence for $\GSp_4$ proved in this paper, one knows that for each place $v$, $\phi_{\Pi,v}$ factors through $\GSp_4(\C)$ with similitude character $\mu_v$. So we have $\phi_{\Pi,v} \cong \phi_{\Pi,v}^{\vee} \otimes \mu_v$.
By the properties of the local Langlands correspondence for $\GL_4$, one deduces that
\[ \Pi_v \cong \Pi_v^{\vee}  \otimes \mu_v. \]
 This proves (ii).
\vskip 5pt

Finally we come to (iii). By the functoriality of the theta correspondence for unramified representations, we see that the degree 6 (partial) standard $L$-function of $\Theta(\pi) = \Pi \boxtimes \mu$ admits a factorization:
\[  L^S(s, \Theta(\pi), std) = \zeta^S(s) \cdot L^S(s, \pi, std). \]
Since $L^S(s, \Theta(\pi),std) = L^S(s, \Pi, \bigwedge^2 \otimes \mu^{-1})$ and $L^S(s,\pi,std)$
is nonzero at $s=1$, we see that the twisted exterior square $L$-function $L^S(s, \Pi, \bigwedge^2 \otimes \mu^{-1})$ has a pole at $s=1$ for any $\Pi \boxtimes\mu$ in the image of the global theta lift.
\vskip 5pt

To complete the proof of (iii),  suppose first that $\Pi$ is a cuspidal representation of $\GL_4(\A)$ such that
$L^S(s, \Pi, \bigwedge^2 \otimes \mu^{-1})$ has a pole at $s=1$. By a result of Jacquet-Shalika \cite{JS},
this is equivalent to $\Pi$ having nonzero Shalika period with respect to $\mu$.
Now let us consider the global theta lift of
$\Pi \boxtimes \mu$ to $\GSp_4(\A)$, which one can easily check to be cuspidal. If one computes the Whittaker-Fourier coefficient of $\Theta(\Pi \boxtimes\mu)$, one obtains an expression involving the Shalika period of $\Pi$ with respect to $\mu$. A proof of this can be found in \cite{So} and \cite[Prop. 3.1]{GT1}. From this, one concludes that $\pi = \Theta(\Pi \boxtimes \mu)$ is a globally generic cuspidal representation of $\GSp_4(\A)$. Thus, $\Pi \boxtimes \mu = \Theta(\pi)$ is in the image of the global theta lift.

\vskip 5pt

On the other hand, if $\Theta(\pi)= \Pi \boxtimes \mu$ is non-cuspidal, then the global theta lift of $\pi$ to $\GSO_{2,2}(\A) $ is nonzero (by the tower property of theta correspondence). It is easy to check that the theta lift of $\pi$
to $\GSO_{1,1}(\A)$ vanishes; indeed, the local theta lift of a representation of $\GSO_{1,1}$ to $\GSp_4$ is never generic. Thus, $\pi = \Theta(\tau_1 \boxtimes \tau_2)$ for a cuspidal representation $\tau_1 \boxtimes \tau_2$ of $\GSO_{2,2}(\A) $ with $\omega_{\pi} = \omega_{\tau_i}$ so that
$\tau_i = \tau_i^{\vee} \otimes \omega_{\pi}$. Moreover, $\tau_1 \ne \tau_2$, for otherwise
the theta lift of $\tau_1 \boxtimes \tau_2$ to $\GSp_4(\A)$ would not be cuspidal.
Conversely, if $\Pi \boxtimes \mu$ is of the type in (iii)(b), so that $\Pi = \tau_1 \boxplus \tau_2$, then
we may consider the global theta lift of the cuspidal representation $\tau_1 \boxtimes \tau_2$ of
$\GSO_{2,2}(\A) $ to $\GSp_4(\A)$. This gives us a globally generic cuspidal representation $\pi = \Theta(\tau_1 \boxtimes \tau_2)$ of $\GSp_4(\A)$. The theta lift of $\pi$ to $\GSO(V)(\A)$ is then equal to $\Pi \boxtimes \mu$.

 \vskip 5pt

This completes the proof of the Theorem.
\qed

\end{document}